\documentclass[12pt,a4paper]{amsart}

\usepackage{tikz-cd} 
\tikzcdset{row sep/normal=3.6em, column
	sep/normal=3.6em}

\usepackage{mathrsfs} 
\usepackage{enumitem} 
\setlist{leftmargin=*}
\usepackage{caption}

\usepackage{stmaryrd}

\newtheorem{theorem}{Theorem}[section]
\newtheorem{proposition}[theorem]{Proposition}
\newtheorem{lemma}[theorem]{Lemma}

\theoremstyle{definition}
\newtheorem{assumption}[theorem]{Assumption}
\newtheorem{construction}[theorem]{Construction}
\newtheorem{convention}[theorem]{Convention}
\newtheorem{corollary}[theorem]{Corollary}
\newtheorem{definition}[theorem]{Definition}
\newtheorem{example}[theorem]{Example}

\newtheorem{notation}[theorem]{Notation}
\newtheorem{openproblem}[theorem]{Open Problem}
\newtheorem*{remark*}{Remark}
\newtheorem{remark}[theorem]{Remark}

\usepackage{newmacros}

\newsavebox{\InnerOne}
\newsavebox{\InnerTwo}
\newsavebox{\InnerTwoEps}
\begin{document}

\title[Quantitative Algebras]{Quantitative Algebras and a Classification of Metric Monads}
\author{J.~Adámek}
\thanks{J.~Adámek and M.~Dostál acknowledge the support by
	the Grant Agency of the Czech Republic under the grant 22-02964S}
\address{Department of Mathematics, Faculty of Electrical Engineering, Czech Technical University in Prague, Czech Republic and Institute for Theoretical Computer Science, Technical University Braunschweig, Germany}
\email{j.adamek@tu-bs.de}
\author{M.~Dostál}
\author{J.~Velebil}
\address{Department of Mathematics, Faculty of Electrical Engineering, Czech Technical University
	in Prague, Czech Republic}
\email{$\{$dostamat,velebil$\}$@fel.cvut.cz}

\begin{abstract}
Quantitative algebras are $\Sigma$-algebras acting on metric spaces, where operations are nonexpanding.
Mardare, Panangaden and Plotkin introduced 1-basic varieties as categories of quantitative algebras presented by quantitative equations.
We prove that for the category $\UMet$ of ultrametric spaces such varieties bijectively correspond to strongly finitary monads on $\UMet$.
The same holds for the category $\Met$ of metric spaces, provided that strongly finitary endofunctors are closed under composition.

For uncountable cardinals $\lambda$ there is an analogous bijection between varieties of $\lambda$-ary quantitative algebras and monads that are strongly $\lambda$-accessible. Moreover, we present a bijective correspondence between $\lambda$-basic varieties as introduced by Mardare et al and enriched, surjections-preserving $\lambda$-accesible monads on $\Met$.
Finally, for general enriched $\lambda$-accessible monads on $\Met$ a bijective correspondence to generalized varieties is presented.
\end{abstract}

\maketitle

\section{Introduction}

Algebras acting on metric spaces, called \emph{quantitative algebras}, have recently received a lot of attention in connection with the semantics of probabilistic computation, see e.g.\ \cite{mpp16,mpp17,bmpp18,mv,bmpp21}.
Mardare, Panangaden and Plotkin worked in~\cite{mpp17} with $\Sigma$-algebras in the category $\Met$ of extended metric spaces and nonexpanding maps, where $\Sigma$ is a signature with finite or countable arities.
They introduced \emph{quantitative equations} which are expressions $t =_{\epsilon} t'$ where $t$ and $t'$ are terms and $\epsilon \geq 0$ is a rational number.
A quantitative algebra $A$ satisfies this equation iff for every interpretation of the variables the computation of $t$ and $t'$ yields elements of distance at most $\epsilon$ in $A$.
Example: almost commutative monoids are quantitative monoids in which, for a given constant $\epsilon$, the distance of $ab$ and $ba$ is at most $\epsilon$.
This is presented by the quantitative equation $xy =_{\epsilon} yx$.

We call classes of quantitative $\Sigma$-algebras that can be presented by a set of quantitative equations \emph{varieties} (in~\cite{mpp17} they are called $1$-basic varieties).
Every variety $\Vvar$ has a free algebra on each metric space, which yields a corresponding monad $\Tmon_\Vvar$ on $\Met$.
Moreover, $\Vvar$ is isomorphic to the Eilenberg-Moore category $\Met^{\Tmon_\Vvar}$.
Can one characterize monads on $\Met$ that correspond to varieties?
One property is that such a monad is \emph{finitary}, i.e.\ it preserves directed colimits.
Another property is that it is enriched, that is, locally nonexpanding.
Unfortunately, these conditions are not sufficient.

We work with \emph{strongly finitary monads} $\Tmon$ introduced by Kelly and Lack~\cite{kelly+lack:strongly-finitary}.
This means that the endofunctor $T$ is the enriched left Kan extension of its restriction to finite discrete spaces.
More precisely, denote by $K : \Set_\fp \to \Met$ the full embedding of finite discrete spaces, then strong finitarity means that $T = \Lan{K}{(T \comp K)}$.
We prove that for every strongly finitary monad $\Tmon$ the Eilenberg-Moore category $\Met^\Tmon$ is concretely isomorphic to a variety of quantitative algebras.
Unfortuanately, we have not been able to solve the following
\begin{openproblem}
	Are all strongly finitary endofunctors on $\Met$ closed under composition?
\end{openproblem}

Assuming an affirmative answer, we present a proof that for every variety $\Vvar$ of quantitative algebras the monad $\Tmon_\Vvar$ is strongly finitary (Corollary~\ref{C:sf}).
Under this assumption we thus get a bijection
\begin{center}
quantitative varieties $\cong$ strongly finitary monads on $\Met$
\end{center}
(Theorem~\ref{T:mainmet}).

Kelly and Lack actually proved, for cartesian closed base categories, the composability of strongly finitary endofunctors.
But $\Met$ is not cartesian closed.
We therefore also investigate quantitative algebras on the category
\[
\UMet
\]
of (extended) ultrametric spaces, a full cartesian closed subcategory of $\Met$.
We then speak about \emph{ultra-quantitative algebras}.
Without any extra assumption we prove a bijection
\begin{center}
utra-quantitative varieties $\cong$ strongly finitary monads on $\UMet$
\end{center}
(Theorem~\ref{T:main}).

For signatures $\Sigma$ with infinitary operations, let $\lambda$ ($\geq \aleph_1$) be a regular cardinal larger than all arities.
We get an immediate generalization: varieties of ultra-quantitative $\Sigma$-algebras correspond bijectively to \emph{strongly $\lambda$-accessible monads} $\Tmon$ on $\UMet$.
This means that $T = \Lan{K}{(T \comp K)}$ for the full embedding $K: \Set_\lambda \hookrightarrow \Met$ of discrete spaces of cardinality less than $\lambda$ (Corollary~\ref{C:dualeq}).
But for $\lambda \geq \aleph_1$ we obtain more: bijective correspondences between
\begin{enumerate}
	\item enriched $\lambda$-accessible monads and generalized $\lambda$-ary varieties (Theorem~\ref{T:main3}), and
	\item $\lambda$-basic (i.e.\ enriched, surjections-preserving and $\lambda$-accessible) monads and $\lambda$-varieties of Mardare et al.\ (Theorem~\ref{T:main2}).
\end{enumerate}
Unfortunately, none of those two results holds for $\lambda = \aleph_0$, as Example~\ref{E:contra} demonstrates.

\vfill
\begin{figure*}[h!]
\begin{tikzpicture}[overlay]
	\filldraw[fill=white,draw=black] (-0.5,1.5) rectangle (5,-2.5);
	\filldraw[fill=white,draw=black] (6.5,1.5) rectangle (12.5,-2.5);
	\filldraw[fill=white,draw=black] (-0.2,0) rectangle (4.8,-2.3);
	\filldraw[fill=white,draw=black] (6.8,0) rectangle (12.3,-2.3);
	\filldraw[fill=white,draw=black] (-0.1,-1.3) rectangle (4.7,-2.1);
	\filldraw[fill=white,draw=black] (6.9,-1.3) rectangle (12.2,-2.1);
\end{tikzpicture}
\begin{tikzcd}[column sep=small, row sep=tiny]
{\text{\footnotesize Enriched monads}} & & {\text{\footnotesize Classes of ultra-quantitative algebras}} \\
{\lambda\text{-accessible}} & & \parbox{4cm}{\centering generalized varieties of $\lambda$-ary algebras} \\
\parbox{4cm}{\centering $\lambda$-accessible + preserve surjections}
& {\quad\cong} & {\lambda\text{-varieties}} \\
{\text{strongly }\lambda\text{-accessible}} & & {\text{varieties of }\lambda\text{-ary algebras}}
\end{tikzcd}
\vspace{0.5cm}
\caption*{A Classification of Monads on $\UMet$ (for $\lambda \geq \aleph_1$)}
\end{figure*}
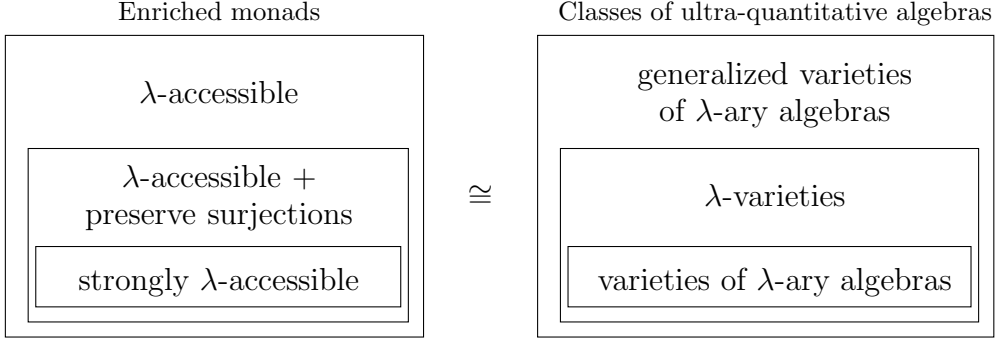

\vspace{0.5cm}

\textbf{Related results}
The correspondence (1) above is based on the presentation of enriched finitary monads due to Kelly and Power~\cite{kelly+power:adjunctions} that we recall in Section~\ref{S:am}.
The correspondence (2) has been established, for $\lambda = \aleph_1$, in~\cite{adamek:varieties-quantitative-algebras}.
We just indicate the (very minor) modifications needed for $\lambda > \aleph_1$ in Section~\ref{S:basic}.
Enriched $\aleph_1$-accessible monads on $\Met$ have also been semantically characterized by Rosický~\cite{rosicky:metric-monads} who uses a different concept of a type of algebras. In Remark~4.11 (2) he indicates a relationship to quantitative algebras.
For cartesian closed base categories Bourke and Garner~\cite{bourke-garner:monads-and-theories} present a bijective correspondence between strongly $\lambda$-ary monads and varieties, where the syntax of the latter is different from the quantitative equations of Mardare et al.

For ordered $\Sigma$-algebras closely related results have been presented recently.
Strongly finitary monads on $\Pos$ bijectively correspond to varieties of finitary ordered $\Sigma$-algebras: see~\cite{kurz-velebil:exactness}, for a shorter proof see~\cite{adv}.

\vspace{0.5cm}

\textbf{Acknowledgement}
The authors are grateful to Jiří Rosický who helped us substantially to understand our main result. And to Nathanael Arkor and John Bourke for suggestions that have improved the presentation of our paper.

\section{The Category of Metric Spaces}

We recall here basic properties of the category $\Met$ needed throughout our paper.
Objects are metric spaces, always meant in the extended sense: the distance $\infty$ is allowed.
Morphisms in $\Met$ are the nonexpanding maps $f : X \to Y$: we have $d(x, x') \geq d(f(x), f(x'))$ for all $x, x' \in X$.

\begin{notation}
\label{N:met}
\phantom{phantom}
\begin{enumerate}
	\item Given spaces $A$ and $B$, the \emph{tensor product} $A \tensor B$ is the cartesian product of the underlying sets equipped with the sum metric:
	\[
	d((a,b), (a',b')) = d(a,a') + d(b,b').
	\]
	With $\One$ the singleton space, $\Met$ is a symmetric monoidal closed category. The internal hom-functor $[A,B]$ is given by the space $\Met(A,B)$ of all morphisms from $A$ to $B$ equipped with the supremum metric:
	\[
	d(f,f') = \sup_{x \in X} d(f(x), f(x')) \quad \text{ for } f, f': A \to B
	\]
	\item The (categorical) product $A \times B$ is the cartesian product of the underlying sets equipped with the maximum metric:
	\[
	d((a,b), (a',b')) = \max \{ d(a,a'), d(b,b') \}.
	\]
	
	\item $\UMet$ denotes the full subcategory of \emph{ultrametric} spaces.
	These are the (extended) metric spaces satisfying the stronger triangle inequality:
	\[
	d(x,z) \leq \max \{ d(x,y), d(y,z) \}.
	\]
	This category is cartesian closed.
	
	\item $\CUMet$ denotes the full subcategory of complete (extended) ultrametric spaces.
	This category is also cartesian closed.

\end{enumerate}
\end{notation}

\begin{remark}[Peter Johnstone, private communication]
The categories $\Met$ and $\CMet$ are not cartesian closed.
This can be derived from the description of exponential objects due to Clementino and Hofmann~\cite{CH}.
Here we present an elementary argument: the endofunctor $(\blank) \times Y$ for
\[
Y = \{ a, b \} \text{ with } d(a,b) = 2
\]
does not have a right adjoint because it does not preserve coequalizers.

Let $X = \{ -3, -2, 2, 3 \}$ be the subspace of the real line.
Consider morphisms $f, g : \One \to X$ representing $-3$ and $3$, respectively.
Their coequalizer $q : X \to Q$ is the quotient map with $Q = \{ -2, 2, x \}$ where $d(-2,x) = 1 = d(2, x)$ and $d(-2,2) = 2$.
Thus in $Q \times X$ all distances are at most $2$.
In contrast, the pair $f \times Y$, $g \times Y$ has a coequalizer $\ol{q} : X \times Y \to \ol{Q}$ which maps the pair $(2,a)$ and $(-2,b)$ to elements of distance $3$.
\end{remark}

\begin{convention}
\phantom{phantom}
\begin{enumerate}
	\item Throughout the rest of this paper all categories $\K$ are understood to be enriched over $\Met$, $\UMet$ or $\CUMet$: every hom-set $\K(X,Y)$ carries a metric and composition is a nonexpanding map from $\K(X,Y) \tensor \K(Y,Z)$ to $\K(X,Z)$.
	For a category $\K$ we denote by $\K_\ordi$ the underlying ordinary category.
	
	Also all functors $F: \K \to \LL$ are understood to be enriched over $\Met$, $\UMet$ or $\CUMet$: the derived maps from $\K(X,Y)$ to $\LL(FX,FY)$ are nonexpanding.
	\item Every ordinary natural transformation between enriched functors is enriched.
	Thus a monad on $\Met$ is enriched iff its underlying endofunctor is.
	Analogously for $\UMet$ or $\CUMet$.
	\item Every set is considered to be the \emph{discrete space}: all distances are $0$ or $\infty$.
	\item The underlying set of a metric space $X$ is denoted by $|X|$.
\end{enumerate}
\end{convention}

\begin{remark}
\label{R:weight}
\phantom{phantom}
\begin{enumerate}
	\item We recall the concept of weighted colimit in a category $\K$.
	Given a diagram $D: \D \to \K$ (a functor with $\D$ small) and a weight (a functor $W: \D^\op \to \Met$) the \emph{weighted colimit}
	\[
	C = \Colim{W}{D}
	\]
	is an object of $\K$ together with an isomorphism
	\[
	\psi_X : \K(C,X) \to [\D^\op,\Met](W,\K(D\blank,X))
	\]
	natural in $X \in \K$.
	\item \emph{Conical colimits} are the special case where $W$ is the trivial weight, constant with value $\One$.
	\item When $\D$ is a $\lambda$-directed poset (every subset of power less than $\lambda$ has an upper bound), we say that $\Colim{W}{D}$ is a \emph{$\lambda$-directed colimit}.
	(In particular, $\D$ is non-empty.)
	\item The \emph{unit} of $C = \Colim{W}{D}$ is the natural transformation $\nu: W \to \K(D\blank,C)$ given by $\nu = \psi_C(\id_C)$.
	The unit determines $\psi$: for every morphism $h : C \to X$ in $\K$ we obtain a natural transformation $\wh{h}: \K(D\blank,C) \to \K(D\blank,X)$ by post-composing with $h$.
	The naturality of $\psi$ implies that
	\[
	\psi_X(h) = W \xrightarrow{\nu} \K(D\blank,C) \xrightarrow{\wh{h}} \K(D\blank,X).
	\]
	\item A functor $T: \K\to \K'$ \emph{preserves} the weighted colimit if $TC = \Colim{W}{TD}$ with the unit having components $T \nu_d$ ($d \in \D$).
	\item  A category is called \emph{cocomplete} if it has weighted colimits.
	\item The dual concept of a weighted limit works with diagrams $D: \D \to \K$ and weights $W: \D\to \K$.
	The category $\K$ is \emph{complete} if it has weighted limits.
\end{enumerate}
\end{remark}

\begin{proposition}[\cite{adamek+rosicky:approximate-injectivity}, Theorem 4.6]
\label{P:coco}
The category $\Met$ is complete and cocomplete.
\end{proposition}

We now present a characterisation of directed colimits in $\Met$.
First we give an example demonstrating that directed colimits are not formed on the level of the underlying sets.

\begin{example}
Let $(A_n)_{n<\omega}$ be the $\omega$-chain of two-element spaces $A_n$ with the underlying set $\{0, 1\}$ and the metric of $A_n$ given by
\[
d(0,1) = 2^{-n}
\]
for all $n < \omega$.
Define the connecting maps to be the identity maps on $\{ 0, 1 \}$.
The colimit of this diagram is the singleton space.
\end{example}

\begin{definition}
\label{D:pres}	
An object $A$ of a category $\K$ is
\begin{enumerate}
	\item \emph{$\lambda$-presentable} if the hom-functor $\K_\ordi(A,\blank): \K_\ordi \to \Set$ preserves $\lambda$-directed colimits, and
	\item \emph{$\lambda$-presentable in the enriched sense} if the functor $\K(A,\blank): \K \to \Met$ preserves $\lambda$-directed colimits.
	\item A category $\K$ is locally $\lambda$-presentable in the enriched sense provided that it has weighted colimits and a set $\K_\lambda$ of objects $\lambda$-presentable in the enriched sense such that $\K$ is the closure of $\K_\lambda$ under $\lambda$-directed colimits.
\end{enumerate}
\end{definition}

\begin{corollary}
If a metric space $M$ is finitely presentable, then $M = \emptyset$.
\end{corollary}

Indeed, the previous example shows that the ordinary forgetful functor $|\blank|: \Met_\ordi \to \Set_\ordi$ does not preserve colimits of $\omega$-chains.
Since $|\blank| \cong \Met(\One,\blank)_\ordi$, we see that $\One$ is not finitely presentable.
An analogous argument holds for every nonempty space.

\begin{proposition}
\label{P:pres}
Let $(D_i)_{i \in I}$ be a directed diagram in $\Met$.
A cocone $c_i: D_i \to C$ is a colimit iff
\begin{enumerate}
	\item it is collectively surjective:
	\[
	C = \bigcup_{i \in I} c_i[D_i],
	\]
	and
	\item for every $i \in I$, given $y, y' \in D_i$ we have
	\[
	d(c_i(y), c_i(y')) = \inf_{j \geq i} d(f_j(y), f_j(y'))
	\]
	where $f_j : D_i \to D_j$ is the connecting morphism.
\end{enumerate}
\end{proposition}
\begin{proof}
The necessity of (1) and (2) is Lemma~2.4 in~\cite{adamek+rosicky:approximate-injectivity}.

To prove sufficiency, assume (1) and (2).
Given a cocone
\[
h_i : D_i \to X \quad (i \in I)
\]
we define $h: |C| \to |X|$ by
\[
h(x) = h_i(y) \text{ whenever } x = c_i(y).
\]
This is well defined: by (1) such $i$ and $y$ exist, and by (2) the value $h_i(y)$ is independent of the choice of $i$ and $y$.
For the independence on $i$ use that $I$ is directed.
For the independence on $y$, we now verify that given $y,y' \in D_i$, then $c_i(y) = c_i(y')$ implies $h_i(y) = h_i(y')$.
Indeed, we show that $d(h_i(y),h_i(y')) < \eps$ for an arbitrary $\eps> 0$.
Apply (2) to $c_i(y) = c_i(y')$: there exists $j \geq i$ with $d(f_j(y), f_j(y')) < \eps$.
Since $h_i = h_j \comp f_j$ (by compatibility) and $h_j$ is nonexpanding, we get
\[
d(h_i(y), h_i(y')) = d(h_j \comp f_j(y), h_j \comp f_j(y')) < \eps.
\]
Thus $h$ is well defined, and by its definition we have $h \comp c_i = h_i$ for every $i \in I$.
Moreover, it is nonexpanding.
Indeed, given $x,x' \in C$ find $i \in I$ and $y,y' \in D_i$ with $x = c_i(y)$ and $x' = c_i(y')$.
Then since $h_i$ and $c_i$ are nonexpanding, we get
\[
d(h(x),h(x')) = d(h_i(y), h_i(y')) \geq d(y,y') \geq d(x,x').
\]
It is easy to see that $h: C \to X$ is the unique morphism with $h \comp c_i = h_i$ for every $i \in I$.
\end{proof}

\begin{corollary}
Directed colimits in $\UMet$ are characterized by the above conditions (1) and (2).

Indeed, this subcategory is closed under directed colimits in $\Met$.
\end{corollary}

\begin{remark}
\label{R:new}
Directed colimits in $\CUMet$ are described analogously, see Proposition~\ref{P:dircol}.
\end{remark}

\begin{example}
\label{E:dirn}
\phantom{phantom}
\begin{enumerate}
	\item For every natural number $n$ the endofunctor $(\blank)^n$ of $\Met$ preserves directed colimits.
	Indeed, if a cocone $(c_i)$ of $D$ satisfies (1) and (2) of the above Proposition, then $c_i^n$ clearly satisfies (1) for $D^n$.
	To verify (2), let $y,y' \in D_i^n$ be given, then for every $\eps > 0$ we prove
	\[
	d(c_i^n(y),c_i^n(y')) < \inf_{j \geq i} d(f_j^n(y),f_j^n(y')) + \eps.
	\]
	Indeed for $y = (y_k)_{k < n}$ and $y' = (y_k')_{k < n}$, given $k < n$, condition (2) for $D$ implies that there exists $j \geq i$ with
	\[
	d(c_i(y_k),c_i(y_k')) < d(f_j(y_k),f_j(y_k')) + \eps.
	\]
	Moreover, since $I$ is directed, we can choose our $j$ independent of $k$.
	Then we get
	\begin{align*}
	d(c_i^n(y),c_i^n(y')) & = \max_{k < n} d(c_i(y_k),c_i(y_k')) \\
	& < \max_{k < n} d(f_j(y_k),f_j(y_k')) + \eps \\
	& = d(f_j^n(y),f_j^n(y')) + \eps.
	\end{align*}
	The same holds for $\UMet$ and $\CUMet$.
	
	\item More generally, given a regular cardinal $\lambda$, the endofunctors $(\blank)^n$ preserve $\lambda$-directed colimits for all cardinals $n < \lambda$.
\end{enumerate}
\end{example}

Whereas directed colimits in $\Met$ are not $\Set$-based, countably directed ones (where every countable subset of $\D$ has an upper bound) are.
The following is also a consequence of 2.6(1) in~\cite{adamek+rosicky:approximate-injectivity}.

\begin{corollary}
\label{C:count}
The ordinary forgetful functor
\[
|\blank| : \Met_\ordi \to \Set_\ordi
\]
preserves countably directed colimits.
\end{corollary}

\begin{proposition}[\cite{adamek+rosicky:approximate-injectivity}, 2.6]
\label{P:presen}
For every regular cardinal $\lambda > \aleph_0$ a space is $\lambda$-presentable in $\Met$ (in the ordinary or enriched sense) iff it has cardinality smaller than $\lambda$.
\end{proposition}

\begin{remark}
	\label{R:fp}
	A metric space is finitely presentable in the enriched sense iff it is finite and discrete. See~\cite{rosicky:metric-monads}, 2.5.
\end{remark}

\begin{definition}[\cite{kelly:structures}]
\label{D:present}
A monoidal closed category $\K$ is \emph{locally $\lambda$-presentable as a closed category} if its underlying category $\K_\ordi$ is locally $\lambda$-presentable, the unit object $I$ is $\lambda$-presentable, and the tensor product of $\lambda$-presentable objects is $\lambda$-presentable.
\end{definition}

\begin{corollary}
\label{C:present}
The category $\Met$ is locally $\lambda$-presentable as a closed category for every regular cardinal $\lambda > \aleph_0$.
Indeed, Proposition~\ref{P:coco} and~\ref{P:presen} imply that $\Met_\ordi$ is locally $\lambda$-presentable: every space is the $\lambda$-filtered colimit of all its subspaces of cardinality smaller than $\lambda$.
And the rest also follows from Proposition~\ref{P:presen}.
\end{corollary}

We now define a weight that will allow to present every metric space as a weighted colimit of \emph{discrete} metric spaces.
This will enable us to characterize strongly finitary functors as functors preserving certain colimits.

\begin{notation}
	\label{N:fol-w}
	Let us denote by
	\[
	B : \B^\op \to \UMet
	\]
	the following weight, called \emph{basic}.
	The category $\B$, which is discretely enriched, is obtained from the linearly ordered set of all real numbers $\eps > 0$ by adding two (parallel) cocones to it with codomain $0$, denoted by
	\[
	l_\eps, r_\eps: \eps \to 0.
	\]
	The functor $B$ is defined on objects by
	\[
	B0 = \{ a \} \text{ and } B\eps = \Two_\eps
	\]
	where $\Two_\eps = \{ 0, 1\}$ with $d(0,1) = \eps$.
	It assigns to $\eps \to \eps'$ the inclusion map $\Two_\eps \hookrightarrow \Two_{\eps'}$.
	For every $\eps > 0$ we put
		\begin{center}
			$Bl_\eps(a) = 0$ and $Br_\eps(a) = 1$
		\end{center}
		
		\begin{lrbox}{\InnerOne}
			\begin{tikzpicture}[overlay]
				\filldraw[fill=white,draw=black] (0,-0.2) rectangle (0.8,0.5);
			\end{tikzpicture}
			\begin{tikzcd}[column sep=10pt]
				\bullet a
			\end{tikzcd}
		\end{lrbox}
		
		\begin{lrbox}{\InnerTwo}
			\begin{tikzpicture}[overlay]
				\filldraw[fill=white,draw=black] (0,-0.2) rectangle (1.4,0.5);
			\end{tikzpicture}
			\begin{tikzcd}[column sep=10pt]
				\bullet \ar[r, no head] & \bullet
			\end{tikzcd}
		\end{lrbox}
		
		\begin{lrbox}{\InnerTwoEps}
			\begin{tikzpicture}[overlay]
				\filldraw[fill=white,draw=black] (0,-0.2) rectangle (1.4,0.5);
			\end{tikzpicture}
			\begin{tikzcd}[column sep=10pt]
				\bullet \ar[r, no head, "\eps"] & \bullet
			\end{tikzcd}
		\end{lrbox}
		
		\begin{equation}
			\label{eq:folia}
			\begin{tikzcd}
				{\eps} \arrow[rd, bend left=10, "r_\eps"] \arrow[rd, bend right=10, swap, "l_\eps"] \arrow[r] & {\eps'} \arrow[d, bend left=10] \arrow[d, bend right=10, swap] \arrow[r] & {\dots} \arrow[ld, bend left=10] \arrow[ld, bend right=10] \\
				& 0 &
			\end{tikzcd}
			\qquad \qquad
			\begin{tikzcd}[column sep=tiny]
				\usebox{\InnerTwoEps} & \usebox{\InnerTwo} & {\usebox{\InnerTwo} \; \; \; \dots} \\
				& \usebox{\InnerOne}
				\arrow[ul, bend left=10, "B l_\eps"] \arrow[ul, bend right=10, swap, "B r_\eps"] \arrow[u, bend left=10] \arrow[u, bend right=10] \arrow[ur, bend left=10] \arrow[ur, bend right=10]
				&
			\end{tikzcd}
		\end{equation}

\end{notation}

\begin{definition}
\label{D:prec}
	A \emph{foliation} of a metric space $M$ is the diagram $D_M : \B \to \Met$ weighted by the basic weight, where
\begin{enumerate}
	
	\item
	\[
	D_M 0 = |M|
	\]
	(the discrete space underlying $M$)
	
	\item For every $\eps > 0$ put $D_M \eps = \Delta_\eps M$, where
	\[
	\Delta_\eps M = \{ (x',x) \in |M| \times |M| \mid d(x',x) \leq \eps \},
	\]
	considered as a discrete space
	
	\item
	For every $\eps > 0$
	\[
	D_M l_\eps : \Delta_\eps M \to |M|
	\]
	is the left-hand projection $(x',x) \mapsto x'$, and
	\[
	D_M r_\eps : \Delta_\eps M \to |M|
	\]
	is the right-hand one, $(x',x) \mapsto x$.
	
\end{enumerate}
\end{definition}

\begin{proposition}
\label{P:prec}
Every metric space $M$ is the colimit of its foliation:
\[
M = \Colim{B}{D_M}.
\]
\end{proposition}

\begin{proof}
Let $X$ be an arbitrary space.
To give a natural transformation $\tau: B \to \Met(D_M\blank, X)$ means to specify a map $f = \tau_0 (a) : |M| \to X$ and maps $\tau_\eps(l), \tau_\eps(r): \Delta_\eps M \to X$ making the following triangles
\[
\begin{tikzcd}
	& {\Delta_\eps M} \arrow[ld, swap, "Bl_\eps"] \arrow[rd, "\tau_\eps(0)"] & & & {\Delta_\eps M} \arrow[ld, swap, "Br_\eps"] \arrow[rd, "\tau_\eps(1)"] & \\
	{|M|} \arrow[rr, swap, "f"] & & X & {|M|} \arrow[rr, swap, "f"] & & X
\end{tikzcd}
\]
commutative.
Thus $\tau$ is determined by $f$.
Moreover, $f : M \to X$ is nonexpanding: given $(x,x') \in \Delta_\eps M$ we have
\begin{align*}
	d(f(x),f(x')) & = d(f \comp Bl_\eps(x,x'), f \comp Br_\eps(x,x')) \\
	& = d(\tau_\eps(0), \tau_\eps(1)) \\
	& \leq \eps
\end{align*}
since $\tau_\eps$ is nonexpanding.
Conversely, every morphism $f : M \to X$ defines a natural transformation $\tau$ by $\tau_0(a) = f$, $\tau_\eps(0) = f \comp B l_\eps$, and $\tau_\eps(1) = f \comp B r_\eps$.
It is easy to see that given another transformation $\tau'$ and putting $f' = \tau_0'(a)$ we have
\[
d(\tau,\tau') = d(f,f').
\]

The desired natural isomorphism
\[
\psi_X: \Met(M,X) \to [\B^\op, \Met](B,\Met(D_M \blank,X))
\]
is given by $\psi_X(f) = \tau$ for the unique natural transformation with $\tau_0(a) = f$.

\end{proof}

\begin{corollary}
	\label{C:fol-pow}
	For every cardinal $\lambda$, the colimits of foliations commute with power to $\lambda$:
	\[
	M^\lambda = \Colim{B}{(D_M)^\lambda}.
	\]
	
\end{corollary}

This follows easily from $\Delta_\eps M^n = (\Delta_\eps M)^n$.

Proposition~\ref{P:prec} allows us to give a `nice' colimit reconstruction of $\Met$ from discrete spaces of bound cardinality.
The idea is to \emph{first} use colimits of foliations to obtain all ultrametric spaces of bound cardinality and \emph{then} to use directed colimits to obtain all ultrametric spaces.
We will formulate this precisely in Section~\ref{S:saf}.

\begin{remark}
\label{R:prec}
\phantom{phantom}
\begin{enumerate}
	\item Denote by $i_M: |M| \to M$ the identity map.
	The unit of the above colimit $M = \Colim{B}{D_M}$ (see Remark~\ref{R:weight} (4)) is the natural transformation
	\[
	\nu: B \to [D_M \blank, M]
	\]
	with components $\nu_0 : a \mapsto i_M$ and
	\[
	\nu_\eps: l \mapsto \pi_l \text{ and } r \mapsto \pi_r.
	\]
	\item In the following we use the factorisation system $(\EE,\MM)$ on $\Met$, where $\EE$ is the class of surjective morphisms and $\MM$ is the class of isometric embeddings.
	Every morphism $f: X \to Y$ factorises via a surjective morphism through $f[X]$ with the metric inherited from $Y$. (The diagonal fill-in is clear.)

\end{enumerate}
\end{remark}

\begin{proposition}
Let $M$ be a metric space.
An enriched functor $H: \Met\to \Met$ preserves the colimit $M = \Colim{B}{D_M}$ iff 
\begin{enumerate}[label=(\roman*)]
	
	\item
	$Hi_M$ is an epimorphism.
	
	\item
	Every morphism $f : H|M| \to X$ that satisfies the following condition
	\begin{equation}
		\label{eq:Tpres}
		d(f \comp Hl_\eps, f \comp Hr_\eps) \leq \eps \text{ for all $\eps > 0$}
	\end{equation}
	factorizes through $Hi_M$.
	
\end{enumerate}
\[
\begin{tikzcd}
	HD_M \eps \ar[r, shift left=1, "Hl_\eps"] \ar[r, shift right=1, swap, "Hr_\eps"] & H|M| \ar[r, "Hi_M"] \ar[rd, swap, "f"] & HM \ar[d, dashed, "\ol{f}"] \\
	& & X
\end{tikzcd}
\]
\end{proposition}

\begin{proof}
	
\begin{enumerate}[label=(\alph*)]
	
	\item 
	We first observe that natural transformations $\tau: B \to [HD_M \blank, X]$ correspond bijectively to morphisms $f : H|M| \to X$ satisfying~\eqref{eq:Tpres}.
	Indeed, put
	\[
	f = \tau_0(a) : H|M| \to X.
	\]
	Then $f$ determines $\tau$ due to the naturality squares:
	\[
	\begin{tikzcd}
		{\{ a \}} \ar[r, "Bl_\eps"] \ar[d, swap, "\tau_0"] & {\Two_\eps} \ar[d, "\tau_\eps"] & {\{ a \}} \ar[d, "\tau_0"] \ar[l, swap, "Br_\eps"] \\
		{[H|M|,X]} \ar[r, swap, "(\blank) \comp HBl_\eps"] & {[HM, X]} & {[H|M|, X]} \ar[l, "(\blank) \comp HBr_\eps"]
	\end{tikzcd}
	\]
	They yield
	\begin{equation}
		\label{eq:rov}
		\tau_\eps(0) = f \comp HBl_\eps \text{ and } \tau_\eps(1) = f \comp HBr_\eps.
	\end{equation}
	Moreover, $f$ satisfies~\eqref{eq:Tpres}:
	\[
	d(f \comp HBl_\eps, f \comp HBr_\eps) \leq d( \tau_\eps \comp B l_\eps, \tau_\eps \comp B r_\eps ) \leq d(B l_\eps, B r_\eps) = \eps.
	\]
	Conversely, if a morphism $f : H|M| \to X$ satisfies~\eqref{eq:Tpres}, then it determines a natural transformation $\tau$ with $\tau_0(a) = f$ via~\eqref{eq:rov}.
	
	\item 
	Conditions (i) and (ii) are sufficient.
	Indeed, $HM = \Colim{B}{HD_M}$ holds since we have an isomorphism
	\[
	\psi_X : [HM,X] \to [\B^\op, \Met]([HD_M\blank,X])
	\]
	natural in $X$ as follows.
	Given a morphism $h : HM \to X$ then $h \comp Hi_M : H|M| \to X$ determines a natural transformation $\tau$ with $\tau_0(a) = h \comp H i_M$.
	We put $\psi(h) = \tau$.
	The morphism $h$ is nonexpanding, thus $h \comp H i_M$ fulfils~\eqref{eq:Tpres}.
	
	Since $H i_M$ is an epimorphism, Item~(a) implies that $\psi_X$ is a bijection.
	Moreover, it preserves distances: let $h' : HM \to X$ have distance $\delta$ from $h$, then, since $H i_M$ is dense, $d(h \comp H i_M, h' \comp H i_M) = \delta$.
	If $\tau'$ is the natural transformation corresponding to $h'$, then
	\[
	d(\tau, \tau') \geq d(\tau_0, \tau_0') = \delta.
	\]
	From~\eqref{eq:rov} (for $\tau_\eps$ and $\tau_\eps'$) in Item~(a) we also get $d(\tau_\eps, \tau_\eps') \leq d(\tau^{\phantom{'}}_0,\tau_0') = \delta$.
	Thus $d(\tau, \tau') = \delta$.
	It is easy to see that the isomorphism $\psi_X$ is natural in $X$.
	
	\item 
	Necessity of (i) and (ii).
	Suppose that $H$ preserves the colimit $M = \Colim{B}{D_M}$.
	Then (ii) holds by Item~(a).
	In order to prove that $Hi_M$ is an epimorphism, let $Y$ denote the closure of its image in $HM$. We consider a factorization
	\[
	Hi_M = H|M| \xrightarrow{e} Y \xrightarrow{m} HM
	\]
	where $e$ is an epimorphism and $m$ is a closed subspace embedding.
	We prove that $m$ is invertible.
	Define a natural transformation
	\[
	\sigma : B \to \Met(HD_M\blank, Y) \text{ by } \sigma_0(a) = e.
	\]
	For the unit $\nu$ of $HM = \Colim{D}{HD_M}$ we have the unique morphism $f : HM \to Y$ making the diagram
	\[
	\begin{tikzcd}
		B \ar[r, "\nu"] \ar[rrd, swap, "\sigma"] & {\Met(D_M\blank, M)} \ar[r, "H_{D_M\blank, M}"] & {\Met(HD_M\blank, HM)} \ar[d, "{\Met(HD_M\blank, f)}"] \\
		& & {\Met(HD_M\blank, Y)}
	\end{tikzcd}
	\]
	commutative.
	This implies $m \comp f = \id$ by the universal property of the unit of $HM \cong \Colim{B}{HD_M}$, which is the horizontal row in the above triangle.
	Thus, $m$ is both a subspace embedding and a split epimorphism, hence an isomorphism.
	Thus $Hi_M$ is an epimorphism.

\end{enumerate}

\end{proof}

\begin{remark}
The same result holds for the full subcategory $\CMet$ of $\Met$ on complete spaces.
The proof is analogous.
\end{remark}

\begin{example}
\label{E:Haus}
\begin{enumerate}
	\item The Hausdorff functor $\mathcal{H}$ preserves colimits of foliations.
	Recall that $\mathcal{H}$ assigns to every space $M$ the space $\mathcal{H}M$ of all compact subsets with the following metric $d_\mathcal{H}:$
	\[
	d_\mathcal{H}(A,B) = \max \{  \sup_{a \in A} d(a,B), \sup_{b \in B} d(b,A) \}
	\]
	where $d(a,B) = \inf_{b \in B} d(a,b)$.
	(In particular $d(A,\emptyset) = \infty$ for $A \neq \emptyset$.)
	On morphisms $f : M \to M'$ it is defined by $A \mapsto f[A]$.
	It is our task to prove that every morphism $f : \mathcal{H}|M| \to X$ satisfying~\eqref{eq:Tpres} above is also a morphism $f: \mathcal{H}M \to X$.
	That is, for all $A,B \in \mathcal{H}M$ we are to verify that
	\[
	d_\mathcal{H}(A,B) = \eps \text{ implies } d_\mathcal{H}(f[A],f[B]) \leq \eps.
	\]
	For each $a \in A$, since $d(a,B) \leq \eps$, there exists $b_a \in B$ with $d(a,b_a) \leq \eps$.
	Analogously for each $b \in B$ we have $a^b \in A$ with $d(b,a^b) \leq \eps$.
	The set $P = \{ (a,b_a) \mid a \in A \} \cup \{ (a^b,b) \mid b \in B \}$ lies in $\mathcal{H}(D_M\eps)$ and fulfils $\mathcal{H}\pi_l(P) = A$ and $\mathcal{H}\pi_r(P) = B$.
	Thus, \eqref{eq:Tpres} implies the desired inequality
	\[
	d_\mathcal{H}(f[A],f[B]) = d_\mathcal{H}(f \comp \pi_l [P], f \comp \pi_r [P]) \leq \eps.
	\]
	
	\item The subfunctor $\mathcal{H}_f$ of $\mathcal{H}$ assigning to every space the space of all finite subsets also preserves colimits of foliations, the argument is the same.
\end{enumerate}
\end{example}

\begin{remark}
	The class $\EE$ lies (strictly) inbetween the class of all epimorphisms, which are the dense morphisms, and all strong epimorphisms.
	For example, if $X$ is a non-discrete space, the epimorphism $\id: |X| \to X$ lies in $\EE$ but is not strong.
\end{remark}

\section{Strongly Accessible Functors}
\label{S:saf}

\begin{assumption}
We assume throughout this section that an infinite regular cardinal $\lambda$ is chosen.
\end{assumption}

\begin{notation}
	\label{N:lambda}
	We denote by
	\[
	\Set_\lambda \text{ and } \Met_\lambda
	\]
	full subcategories of $\Met$ representing, up to isomorphism, all sets (discrete spaces) of cardinality less than $\lambda$, or all spaces of cardinality less than $\lambda$, respectively.
	We also denote by
	\[
	K: \Set_\lambda \hookrightarrow \Met \text{ and } K^*: \Met_\lambda \hookrightarrow \Met
	\]
	the full embeddings.
\end{notation}

Recall that an endofunctor of $\Met$ is \emph{$\lambda$-accessible} iff it preserves $\lambda$-directed colimits.
We characterize these functors using the concept of density presentation of Kelly~\cite{kelly:book} that we shortly recall.
Let $L: \A \to \Met$ be a functor with $\A$ small.
This leads to a functor 
\[
	\wt{L} : \Met \to [\A^\op, \Met], \quad X \mapsto \Met(L\blank,X).
\]
A colimit of a weighted diagram in $\Met$ is called \emph{$L$-absolute} if $\wt{L}$ preserves it.
A \emph{density presentation} of $L$ is a collection of $L$-absolute weighted colimits in $\Met$ such that every space can be obtained from finite discrete spaces by an (iterated) application of those colimits.

\begin{proposition}
	An endofunctor $T$ of $\Met$ is $\lambda$-accessible iff it is the left Kan extension of its restriction $T \comp K^*$ to spaces of power less than $\lambda$:
	\[
	T = \Lan{K^*}{(T \comp K^*)}.
	\]
\end{proposition}

\begin{proof}
	The functor $K^* : \Met_\lambda \to \Met$ is dense.
	Indeed, this follows from~\cite{kelly:book}, Theorem~5.19.
	The functor $K^*$ has a density presentation consisting of all $\lambda$-directed diagrams in $\Met_\lambda$.
	(Indeed, every metric space $M$ is the $\lambda$-directed colimit of all of its subspaces of power less than $\lambda$. Now consider the corresponding $\lambda$-directed diagram in $\Met_\lambda$.)
	For every space $M$ in $\Met_\lambda$ the functor $[M,\blank]$ preserves $\lambda$-filtered colimits by Proposition~\ref{P:presen}.
	Thus $\lambda$-filtered colimits are $K^*$-absolute, and our proposition follows from Theorem~5.29 in~\cite{kelly:book}.
\end{proof}

We now characterize strongly $\lambda$-accessible endofunctors on $\Met$.
These are functors $T$ which are (enriched) left Kan extensions of their restrictions to discrete spaces of power larger than $\lambda$.
More precisely, recalling Notation~\ref{N:lambda} the strongly $\lambda$-accessible functors are defined as follows:

\begin{definition}
\label{D:sa}
An endofunctor $T$ of $\Met$ is \emph{strongly $\lambda$-accessible} if it is the left Kan extension of its restriction $T \comp K$ to discrete spaces of power less than $\lambda$:
\[
T = \Lan{K}{(T \comp K)}.
\]
If $\lambda = \aleph_0$, we speak of \emph{strongly finitary endofunctors}.
\end{definition}

This concept of a strongly finitary functor was introduced by Kelly and Lack~\cite{kelly+lack:strongly-finitary} more generally for endofunctors of locally finitely presentable categories as enriched categories.

Recall that the above equality means that there is a natural transformation $\sigma : T \comp K \to T \comp K$ such that for every endofunctor $F$ and every natural transformation $\phi: T \comp K \to T \comp F$ there exists a unique $\ol{\phi}: T \to F$ with $\phi = (\ol{\phi}K) \comp \sigma$.

\begin{example}
\label{E:pown}
\begin{enumerate}
	\item For every natural number $n$ the functor $(\blank)^n$ is strongly finitary.
	Indeed, for every (enriched) endofunctor $F$ to give a natural transformation $\phi: (\blank)^n \comp K \to F \comp K$ means precisely to specify an element of $Fn$ (since $(\blank)^n$ is naturally isomorphic to the hom-functor of $n$).
	And this is the same as specifying a natural transformation $\ol{\phi}: (\blank)^n \to F$.
	Thus, $(\blank)^n$ is the left Kan extension with $\sigma = \id$ as claimed.
	\item Directed colimits and colimits of foliations are $K$-absolute.
	Indeed, the functor $\wt{K}$ assigns to $X$ the functor of finite powers of $X$.
	And directed colimits commute with finite powers by Example~\ref{E:dirn}, whereas reflexive coinserters do by Remark~\ref{R:prec}~(2).
\end{enumerate}
\end{example}

\begin{theorem}
\label{T:sa}
An endofunctor of $\Met$ is strongly $\lambda$-accessible iff it is $\lambda$-accessible and preserves colimits of foliations.
\end{theorem}

\begin{proof}
\phantom{phantom}

The functor $K: \Set_\lambda \to \Met$ has a density presentation consisting of all $\lambda$-directed diagrams and all foliations of spaces in $\Met_\lambda$. Indeed:
	\begin{enumerate}
		\item Every space in $\Met_\lambda$ is a colimit of its foliation (Proposition~\ref{P:prec}) which is a weighted diagram in $\Set_\lambda$;
		\item Every metric space $M$ is a $\lambda$-directed colimit of a diagram in $\Met_\lambda$: consider all subspaces of $M$ of cardinality less than $\lambda$, and form the corresponding diagram in $\Met_\lambda$.
	\end{enumerate}
Moreover, we know that both of these types of colimits are $K$-absolute.
Our theorem then follows from Theorem~5.29 in~\cite{kelly:book}: an endofunctor $T$ of $\Met$ is strongly $\lambda$-accessible iff it preserves
\begin{enumerate}[label=(\alph*)]
	\item $\lambda$-directed colimits and
	\item colimits of foliations of spaces in $\Met_\lambda$.
\end{enumerate}
Thus it only remains to observe that (a) and (b) imply that $T$ preserves all colimits of foliations.
Indeed, for every space $M$ let $M_i$ ($i \in I$) be the $\lambda$-directed diagram of all subspaces of cardinality less than $\lambda$.
From $M_i = \Colim{B}{D_{M_i}}$ (Proposition~\ref{P:prec}) we derive $TM = \Colim{B}{D_M}$:
the unit $\nu$ of the latter colimit is namely the $\lambda$-directed colimit of the units $\nu_i$ w.r.t\ $M_i$ ($i \in I$).
Indeed, recall from Remark~\ref{R:prec} that the component $\nu_a$ is given by the assignment $0 \mapsto \id: |M| \to M$.
This is a $\lambda$-directed colimit of the components $(\nu_i)_a$ given by $0 \mapsto \id: |M_i| \to M_i$.
Analogously for the component $\nu_\eps$.
\end{proof}

\begin{remark}
The same result holds for $\CMet$, the proof is analogous.
\end{remark}

\begin{example}
	\label{E:sa}
	\phantom{phantom}
	\begin{enumerate}
		\item Every coproduct of strongly finitary endofunctors is strongly finitary.
		Indeed, more generally, in the category $[\Met,\Met]$ the strongly finitary endofunctors are closed under weighted colimits: this follows from Definition~\ref{D:sa} directly.
		\item Let $\Sigma$ be a finitary signature, i.e., a set of symbols $\sigma$ with prescribed arities $\arity(\sigma) \in \Nat$.
		The corresponding \emph{polynomial functor} $H_\Sigma : \Met \to \Met$ given by
		\[
		H_\Sigma X = \Coprod_{\sigma \in \Sigma} X^{\arity(\sigma)}
		\]
		is strongly finitary.
		This follows from (1) and Example~\ref{E:pown}~(1).
		\item The finite Hausdorff functor $\mathcal{H}_f$ (Example~\ref{E:Haus}) is strongly finitary.
		Indeed, it follows easily from Proposition~\ref{P:pres} that $\mathcal{H_f}$ preserves directed colimits.
		And we know from Example~\ref{E:Haus} that it preserves colimits of foliations.
	\end{enumerate}
\end{example}

\begin{notation}
	\label{N:fsf}
	\phantom{phantom}
	\begin{enumerate}
		
		\item $[\Met,\Met]$ denotes the category of endofunctors and natural transformations.
		It is enriched by the supremum metric: the distance of natural transformations $\alpha, \beta: F \to G$ is
		\[
		\sup_{X \in \Met} d(\alpha_X,\beta_X).
		\]
		
		\item For every category $\K$, given a non-full subcategory of $\K_\ordi$, we can consider it as the (enriched) category $\LL$ by defining $\LL(X,Y)$ as the metric subspace of $\K(X,Y)$.
		In particular, the category
		\[
		\Mon(\Met)
		\]
		of all monads on $\Met$ is enriched as a non-full subcategory of $[\Met,\Met]$: the distance of parallel monad morphisms is the supremum of the distances of their components. Analogously for its full subcategory
		\[
		\Mon_\lambda(\Met)
		\]
		of all $\lambda$-accessible monads.
		(Notice that we are not using the view of monads as monoids here:
		indeed, the ordinary category of monoids in $\V$ is not $\V$-enriched in general.)

		\item We denote by $\Mon_\fp(\Met)$ and $\Mon_\sf(\Met)$ the categories of finitary and strongly finitary monads on $\Met$, respectively, as full subcategories of $\Mon(\Met)$.
		
	\end{enumerate}
\end{notation}

\begin{remark}
\label{R:coref}
\begin{enumerate}
	\item For every endofunctor $F$ of $\Met$ the functor $\ol{F} = \Lan{K}{F\comp K}$ is a strongly finitary reflection of $F$.
	Indeed, the universal natural transformation $\sigma : \ol{F} \comp K \to F \comp K$ yields a unique natural transformation $\tau : \ol{F} \to F$ with $\tau K = \sigma$.
	Given a strongly finitary functor $T = \Lan{K}{T\comp K}$ (with its universal natural transformation $\sigma' : TK \to TK$), for every natural transformation $\alpha : T \to F$ there exists a unique natural transformation $\ol{\alpha} : T \to \ol{F}$ with $\alpha = \tau \comp \ol{\alpha}$.
	Indeed, we have a unique natural transformation $\ol{\alpha} = \Lan{K}{\alpha}: \ol{T} \to \ol{F}$ making the square below commutative
	\[
	\begin{tikzcd}
	{TK = \ol{T}K} \ar[r, "\ol{\alpha} K"] \ar[d, swap, "\sigma'"] & {\ol{F} K} \ar[d, "\sigma = \tau K"] \\
	{TK} \ar[r, swap, "\alpha K"] & {FK}
	\end{tikzcd}
	\]
	From $(\tau \ol{\alpha}) K = (\alpha K) \comp \sigma'$ we conclude $\tau \comp \ol{\alpha} = \alpha$.
	\item Analogously, for every endofunctor $F$ the functor $\Lan{K^*}{FK^*}$ is a finitary reflection of $F$.
\end{enumerate}
\end{remark}

\begin{openproblem}
\label{OP}
Is a composite of strongly finitary endofunctors of $\Met$ strongly finitary?
\end{openproblem}

\begin{remark}
\label{R:comp}
For cartesian closed locally finitely presentable categories, strongly finitary functors do compose, as proved by Kelly and Lack~\cite{kelly+lack:strongly-finitary}.
Thus, the corresponding answer for $\UMet$ or $\CUMet$ is affirmative.
The proof in loc.~cit.\ also works for strongly $\lambda$-ary endofunctors on locally $\lambda$-presentable categories.
\end{remark}

\begin{proposition}
\label{P:lambda-acc}
The category $\Mon_\lambda(\Met)$ of $\lambda$-accessible monads is locally $\lambda$-presentable in the enriched sense, for every regular cardinal $\lambda > \aleph_0$.
\end{proposition}

\begin{proof}
\phantom{phantom}
\begin{enumerate}

	\item The functor
	\[
	W : \Mon_\lambda(\Met) \to [\Met_\lambda,\Met]
	\]
	assigning to every $\lambda$-accessible monad $\Tmon = (T,\eta,\mu)$ the domain restriction of $T$ to $\Met_\lambda$ is $\lambda$-accessible.
	This follows from Theorem~6.18(2) of~\cite{bird:thesis}.
	Let $J : |\Met_\lambda| \to \Met_\lambda$ denote the underlying discrete category of $\Met_\lambda$.
	The functor
	\begin{align*}
		V : [\Met_\lambda,\Met] & \to [|\Met_\lambda|,\Met]
	\end{align*}
	of precomposition with $J$ is also $\lambda$-accessible, in fact, it has a right adjoint.
	Thus, the composite
	\[
	U = V \comp W : \Mon_\lambda(\Met) \to [|\Met_\lambda|,\Met]
	\]
	is $\lambda$-accessible.
	
	\item $U$ is a monadic functor.
	Indeed, Lack proved (\cite{lack:monadicity}, Corollary~3 and Remark~4) that the underlying ordinary functor $U_\ordi$ is monadic.
	Moreover, the ordinary left adjoint to $U_\ordi$ is easily seen to be enriched.
	Hence, we have an enriched adjunction $F \adj U$ with the ordinary adjunction $F_\ordi \adj U_\ordi$ being monadic.
	By~\cite{dubuc}~Theorem~II.2.1 this means that $U$ is monadic (in the enriched sense).
	
	\item Let $\Tmon$ denote the $\lambda$-accessible monad on $[|\Met_\lambda|,\Met]$, associated to $F \adj U$.
	By Corollary~\ref{C:present} $\Met$ is locally $\lambda$-presentable in the enriched sense, hence, so is $[|\Met_\lambda|,\Met]$.
	By~\cite{bird:thesis} Theorem~6.9, the Eilenberg-Moore category of $\Tmon$ is locally $\lambda$-presentable as an enriched category.
	This category is equivalent (in the enriched sense) to $\Mon_\lambda(\Met)$ by (2) above.
	
\end{enumerate}
\end{proof}

\begin{corollary}
\label{C:coco1}
The category $\Mon_{\aleph_1}(\Met)$ is cocomplete.
\end{corollary}

Recall the open problem~\ref{OP}.

\begin{proposition}
\label{P:cls}
Suppose that strongly finitary endofunctors on $\Met$ are closed under composition. Then the category $\Mon_\fp(\Met)$ is cocomplete and $\Mon_\sf(\Met)$ is closed in it under colimits.
\end{proposition}

\begin{proof}
	By Corollary~\ref{C:coco1} it is sufficient to prove that both of the above
	full subcategories of $\Mon_{\aleph_1}(\Met)$ are coreflective in it.
	\begin{enumerate}
		
		\item The coreflectivity of $\Mon_\sf(\Met)$.
		Given a countably accessible monad $\Tmon = (T,\mu,\eta)$, we form the coreflection of $T$ in the category of strongly finitary endofunctors on $\Met$.
		By Remark~\ref{R:coref} this is the functor
		\[
		\ol{T} = \Lan{K}{T \comp K}.
		\]
		Let $\tau: \ol{T} \to T$ be the universal natural transformation.
		We are going to present a canonical monad structure $(\ol{T},\ol{\mu},\ol{\eta})$ making $\tau$ a monad morphism.
		It then follows easily that $\tau : (\ol{T},\ol{\mu},\ol{\eta}) \to (T,\mu,\eta)$ is the desired coreflection of $\Tmon$ in $\Mon_\sf(\Met)$.
		
		The universal property of $\tau$ implies that $\eta : \Id \to T$ factorizes as follows
		\[
		\begin{tikzcd}
		{\Id} \arrow[r, "\ol{\eta}"] \arrow[rd, swap, "\eta"] & {\ol{T}} \arrow[d, "\tau"] \\
		& T 
		\end{tikzcd}
		\]
		for a unique natural transformation $\ol{\eta}$.
		And $\mu : T \comp T \to T$ yileds a unique natural transformation $\ol{\mu} : \ol{T} \comp \ol{T} \to \ol{T}$ making the square below commutative (using that $\ol{T} \comp \ol{T}$ is finitary):
		\[
		\begin{tikzcd}
		\ol{T} \comp \ol{T} \arrow[r, "\ol{\mu}"] \arrow[d, "\tau * \tau", swap] & \ol{T} \arrow[d, "\tau"] \\
		T \comp T \arrow[r, swap, "\mu"] & T
		\end{tikzcd}
		\]
		We verify that $(\ol{T},\ol{\mu},\ol{\eta})$ is a monad.
		
		From the unit law $\mu \comp T\eta = \id$ we derive $\ol{\mu} \comp \ol{T} \ol{\eta} = \id$, using the universal property of $\tau$ and the following commutative diagram:
		\[
		\begin{tikzcd}
		\ol{T} \comp \ol{T} \arrow[rr, "\ol{\mu}"] \arrow[dd, swap, "\ol{T} \tau"]  & & \ol{T} \arrow[ddd, "\tau"] \\
		& \ol{T} \arrow[ul, swap, "\ol{T} \ol{\eta}"] \arrow[dl, "\ol{T} \eta"] \arrow[d, "\tau"] & \\
		\ol{T} \comp T \arrow[d, swap, "\tau T"] & T \arrow[ld, "T\eta"] \arrow[rd, equal] & \\
		T \comp T \arrow[rr, swap, "\mu"] & & T
		\end{tikzcd}
		\]
		Indeed, the left-hand triangle and the outward square commute by definition of $\ol{\eta}$ and $\ol{\mu}$, respectively.
		The remaining left-hand part is the naturality of $\tau$.
		Since $\tau \comp (\ol{\mu} \comp \ol{T} \ol{\eta}) = \tau$, we have $\ol{\mu} \comp \ol{T} \ol{\eta} = \id$.
		Analogously for the unit law $\mu \comp \ol{\eta} \ol{T} = \id$.
		
		From the associative law $\mu \comp T \mu = \mu \comp \mu T$ for $\Tmon$ we derive the one for $\ol{\Tmon}$ using the universal property of $\tau$ and the following commutative diagram:
		\[
		\begin{tikzcd}
		\ol{T} \ol{T} \ol{T} \arrow[rrrr, "\ol{T} \ol{\mu}"] \arrow[rd, "\ol{T} \ol{T} \tau"] \arrow[rdd, swap, "\ol{T} \tau \ol{T}"] \arrow[dddddd, swap, "\ol{\mu} \ol{T}"] & & & & \ol{T} \ol{T} \arrow[dddddd, "\ol{\mu}"] \arrow[ldd, swap, "\ol{T} \tau"] \\
		& \ol{T} \ol{T} T \arrow[rd, "\ol{T} \tau T"] \arrow[rrru, phantom, "\text{(def $\ol{\mu}$)}"] \arrow[d, phantom, "\text{($\tau$)}"] & & & \\
		& \ol{T} T \ol{T} \arrow[r, swap, "\ol{T} T \tau"] \arrow[d, swap, "\tau T \ol{T}"] \arrow[dd, bend right=70, phantom, "\text{(def $\ol{\mu}$)}"] \arrow[rd, phantom, "\text{($\tau$)}"] & \ol{T} T T \arrow[r, "\ol{T} \mu"] \arrow[d, "\tau T T"] \arrow[rd, phantom, "\text{($\tau$)}"] & \ol{T} T \arrow[d, swap, "\tau T"] \arrow[dd, bend left=70, phantom, "\text{(def $\ol{\mu}$)}"] & \\
		& T T \ol{T} \arrow[r, swap, "TT\tau"] \arrow[d, swap, "\mu \ol{T}"] \arrow[rd, phantom, "\text{($\mu$)}"] & T T T \arrow[r, "T \mu"] \arrow[d, swap, "\mu T"] & T T \arrow[d, "\mu"] & \\
		& T \ol{T} \arrow[r, swap, "T \tau"] \arrow[d, phantom, "\text{($\tau$)}"] & T T \arrow[r, swap, "\mu"] & T & \\
		& \ol{T} T \arrow[ur, swap, "\tau T"] \arrow[rrrd, phantom, "\text{(def $\ol{\mu}$)}"] & & & \\
		\ol{T} \ol{T} \arrow[uur, "\tau \ol{T}"] \arrow[ur, swap, "\ol{T} \tau"] \arrow[rrrr, swap, "\ol{\mu}"] & & & & \ol{T} \arrow[uul, swap, "\tau"]
		\end{tikzcd}
		\]
		Indeed, the parts denoted by ($\tau$) and ($\mu$) commute by the naturality of the corresponding transformations.
		Those denoted by (def~$\ol{\mu}$) follow from the definition of $\ol{\mu}$.
		The remaining square is the associativity of $\mu$.
		
		We conclude that $(\ol{T}, \ol{\mu}, \ol{\eta})$ is a monad, and by its definition $\tau$ is a monad morphism.
		
		It remains to verify that every monad morphism
		\[
		\sigma: (S,\mu^S,\eta^S) \to \Tmon
		\]
		with $S$ is finitary factorizes through $\tau$ via a (necessarily unique) monad morphism $\ol{\sigma} : (S, \mu^S, \eta^S) \to \ol{\Tmon}$.
		We have a unique natural transformation
		\[
		\ol{\sigma} : S \to \ol{T} \text{ with } \sigma = \tau \comp \ol{\sigma}.
		\]
		It preserves the unit: we have
		\[
		\tau \comp (\ol{\sigma} \comp \eta^S) = \sigma \comp \eta^S = \eta = \tau \comp \ol{\eta}
		\]
		from which $\ol{\sigma} \comp \eta^S = \ol{\eta}$ follows.
		And it preserves the multiplication: we verify $\ol{\sigma} \comp \mu^S = \ol{\mu} \comp (\ol{\sigma} * \ol{\sigma})$ analogously:
		\begin{align*}
		\tau \comp (\ol{\sigma} \comp \mu^S) & = \sigma \comp \mu^S && \text{def.\ of $\ol{\sigma}$} \\
		& = \mu \comp (\sigma * \sigma) \\
		& = \mu \comp (\tau * \tau) \comp (\ol{\sigma} * \ol{\sigma}) && \text{def.\ of $\ol{\sigma}$} \\
		& = \tau \comp \ol{\mu} \comp (\ol{\sigma} * \ol{\sigma}) && \text{def.\ of $\ol{\mu}$}
		\end{align*}
	
		Since both the full embedding
		\[
		\Mon_\sf(\Met) \to \Mon_{\aleph_1}(\Met)
		\]
		and its (ordinary) right adjoint constructed above are in fact enriched, the category $\Mon_\sf(\Met)$ is coreflective in $\Mon_{\aleph_1}(\Met)$. 
		
		\item The coreflectivity of $\Mon_\fp(\Met)$ is verified completely analogously.
		
	\end{enumerate}
\end{proof}

We analogously denote by
\[
\Mon_\fp(\UMet) \text{ and } \Mon_\sf(\UMet)
\]
the categories of (strongly) finitary monads on $\UMet$.

\begin{proposition}
The category $\Mon_\fp(\UMet)$ is cocomplete and $\Mon_\sf(\UMet)$ is closed in it under colimits.
\end{proposition}

The proof is the same as above, using Remark~\ref{R:comp}.

\begin{proposition}
\label{P:fact}
Let $\Tmon$ be a finitary monad on $\Met$ preserving surjective morphisms.
Every morphism $h: \Tmon \to \Smon$ in the category $\Mon_\fin(\Met)$ has a factorization
\[
\begin{tikzcd}[row sep=scriptsize, column sep=scriptsize]
{\Tmon} \arrow[rr, "h"] \arrow[rd, swap, "e"] & & {\Smon} \\
& {\Tmon'} \arrow[ur, swap, "m"] &
\end{tikzcd}
\]
where the components of $e$ are surjective and those of $m$ are isometric embeddings.
The same result holds in the category $\Mon_\fp(\UMet)$.
\end{proposition}

\begin{proof}
We present a proof for $\Mon_\fin(\Met)$, the case $\Mon_\fin(\UMet)$ is analogous.
\begin{enumerate}[label=(\alph*)]
	
	\item For every metric space $X$ factorize $h_X$ as a surjective morphism $e_X : TX \to T'X$ followed by an isometric embedding $m_X : T'X \to SX$.
	This defines an endofunctor $T'$ whose action on a morphism $f : X \to Y$ is given by the following diagonal fill-in:
	\[
	\begin{tikzcd}
	{TX} \arrow[r, "e_X"] \arrow[d, swap, "Tf"] & {T'X} \arrow[d, "m_X"] \arrow[ldd, swap, dotted, "T'f"] \\
	{TY} \arrow[d, swap, "e_Y"] & {SX} \arrow[d, "Sf"] \\
	{T'Y} \arrow[r, swap, "m_Y"] & {SY}
	\end{tikzcd}
	\]
	It is easy to verify that $T'$ is well defined; the fact that it is enriched follows from $m_Y$ being an isometric embedding.
	The above diagram yields natural transformations $e : T \to T'$ and $m : T' \to S$ with $h = m \comp e$.
	
	We next prove that $e$ and $m$ are monad morphisms for some corresponding monad structure $\mu^{T'} : T'T' \to T'$ and $\eta^{T'} : \Id \to T'$.
	Finally, we prove that $T'$ is finitary.
	
	\item We define a natural transformation
	\[
	\eta^{T'} = \Id \xrightarrow{\eta^T} T \xrightarrow{e} T'.
	\]
	Furthermore, recall that $e * e : T \comp T \to T' \comp T'$ has components
	\[
	TTX \xrightarrow{Te_X} TT'X \xrightarrow{e_{T'X}} T'T'X
	\]
	that are surjective because $T$ preserves surjections.
	We can thus define a natural transformation
	\[
	\mu^{T'} : T'T' \to T'
	\]
	by components using the diagonal fill-in as follows:
	\[
	\begin{tikzcd}
	{TTX} \arrow[r, "(e*e)_X"] \arrow[d, swap, "\mu_X^T"] & {T'T'X} \arrow[d, "(m*m)_X"] \arrow[ldd, swap, dotted, "\mu_X^{T'}"] \\
	{TX} \arrow[d, swap, "e_X"] & {TTX} \arrow[d, "\mu_X^S"] \\
	{T'X} \arrow[r, swap, "m_X"] & {SX}
	\end{tikzcd}
	\]
	We now verify that $(T',\mu^{T'}, \eta^{T'})$ is a monad.
	From the above definitions it then follows that $e$ and $m$ are monad morphisms.
	\begin{enumerate}[label=(\roman*)]
		\item We show that $\mu^{T'} \comp T'\eta^{T'} = \id$ holds.
		This follows from the diagram below:
		
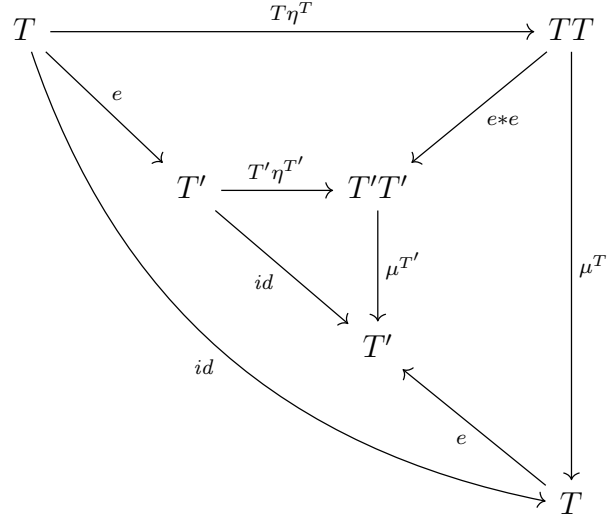
\[
		\begin{tikzcd}
		{T}
		\arrow[rrr, "T\eta^T"]
		\arrow[rd, "e"]
		\arrow[rrrddd, swap, bend right, "\id"]
		& & & {TT}
		\arrow[ddd, "\mu^T"]
		\arrow[ld, "e * e"] \\
		& {T'}
		\arrow[r, "T'\eta^{T'}"]
		\arrow[rd, swap, "\id"]
		& {T'T'}
		\arrow[d, "\mu^{T'}"]
		& \\
		& & {T'} & \\
		& & & {T}
		\arrow[ul, "e"]
		\end{tikzcd}
		\]
		Since $\Tmon$ is a monad, the outward triangle commutes.
		The right-hand part commutes by the definition of $\mu^{T'}$.
		We prove next that the upper part commutes, then it follows that the desired inner triangle commutes when precomposed with $e$.
		Since $e$ has surjective components, this implies that the triangle commutes, too.
		For the upper part use the following diagram
		
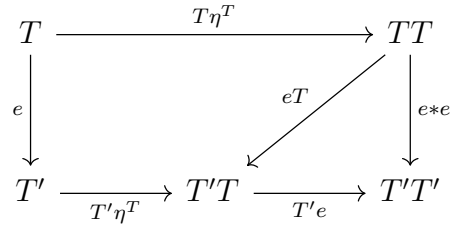
\[
		\begin{tikzcd}
		{T} \arrow[rr, "T\eta^T"] \arrow[d, swap, "e"] & & {TT} \arrow[d, "e * e"] \arrow[ld, swap, "eT"] \\
		{T'} \arrow[r, swap, "T'\eta^T"] & {T'T} \arrow[r, swap, "T'e"] & {T'T'}
		\end{tikzcd}
		\]
		The left-hand part commutes by the naturality of $e$, and the right-hand part is the (equivalent) definition of $e * e = T'e \comp eT$.
		\item The proof of $\mu^{}T' \comp \eta^{T'}T' = \id$ is analogous.
		\item To prove $\mu^{T'} \comp T'\mu^{T'} = \mu^{T'} \comp \mu^{T'}T'$ use the following diagram
		\begin{equation*}
		\label{diag:assoc} \tag{\textasteriskcentered}
		\begin{tikzcd}
		{TTT}
		\arrow[rrr, "T\mu^T"]
		\arrow[rd, "e*e*e"]
		\arrow[ddd, swap, "\mu^T T"]
		& & & {TT}
		\arrow[ddd, "\mu T"]
		\arrow[ld, swap, "e*e"]
		\\
		& {T'T'T'}
		\arrow[r, "T' \mu^{T'}"]
		\arrow[d, swap, "\mu^{T'} T'"]
		& {T'T'}
		\arrow[d, "\mu^{T'}"]
		& \\
		& {T'T'}
		\arrow[r, swap, "\mu^{T'}"]
		& {T} & \\
		{TT}
		\arrow[ur, "e*e"]
		\arrow[rrr, swap, "\mu^T"]
		& & & {T}
		\arrow[ul, swap, "e"]
		\end{tikzcd}
		\end{equation*}
		We prove first that the upper part commutes.
		Since $e*e*e = T'(e*e) \comp eTT$, this follows from the commutative diagram below:
		\[
		\begin{tikzcd}
		{TTT}
		\arrow[r, "T\mu^T"]
		\arrow[d, swap, "eTT"]
		& {TT}
		\arrow[d, "eT"]
		\\
		{T'TT}
		\arrow[r, "T' \mu^T"]
		\arrow[d, swap, "T'(e*e)"]
		& {T'T}
		\arrow[d, "T'e"]
		\\
		{T'T'T'}
		\arrow[r, swap, "T' \mu^{T'}"]
		& {T'T'}
		\end{tikzcd}
		\]
		The upper square commutes by the naturality of $e$, and the lower one does since it is the application of $T'$ to the left-hand part of the definition of $\mu^{T'}$.
		
		Analogously, the left-hand part of the diagram~\eqref{diag:assoc} commutes. The right-hand part and the lower one follow from the definition of $\mu^{T'}$.
		
		The outward square commutes because $\Tmon$ is a monad.
		We conclude that the desired inner square commutes when precomposed by $e * e * e$.
		The latter transformation has surjective components.
		Thus, the middle square commutes.
	\end{enumerate}	

	\item It remains to prove that $T'$ is finitary.
	Let $(c_i)_{i \in I}$ be a colimit cocone of a directed diagram $(D_i)_{i \in I}$.	
	It is our task to prove that $(T'c_i)_{i \in I}$ has properties (1) and (2) of Proposition~\ref{P:pres}.
	\begin{description}
		\item[Property (1)] Given $x \in T'C$ we are to find $i \in I$ with $x \in T'c_i[T'C_i]$.
		As $e_C$ is epic, we have $y \in TC$ with $x = e_C(y)$.
		Since $T$ preserves the colimit of $(D_i)$, there exists $i \in I$ and $y' \in TD_i$ with $y = Tc_i(y')$.
		Thus $x' = e_{D_i}(y')$ fulfils $x = T'c_i(x')$ as desired:
		\[
		\qquad \qquad \qquad x = e_C(y) = e_C \comp Tc_i(y') = T'c_i \comp e_{D_i} (y') = T'c_i(x').
		\]
		\item[Property (2)] Given $i \in I$ and $y_1, y_2 \in T' D_i$, we are to prove
		\[
		d(T'c_i(y_1), T'c_i(y_2)) = \inf_{j \geq i} d(T'f_j(y_1),T'f_j(y_2))
		\tag{\textasteriskcentered} \label{eq:meet}
		\]
		Put $\ol{y}_k = m_{D_i}(y_k)$ for $k = 1,2$ and use that, since $\Tmon$ is a finitary monad, $(Tc_i)$ is a colimit cocone of $(TD_i)$.
		Thus it has the property (2):
		\[
		d(TC_i(\ol{y}_1, Tc_i(\ol{y}_2))) = \inf_{j \geq i} d(Tf_j(\ol{y}_1), Tf_j(\ol{y}_2)).
		\tag{\textasteriskcentered\textasteriskcentered} \label{eq:meet1}
		\]
		The left-hand sides of \eqref{eq:meet} and \eqref{eq:meet1} are equal: use that $m_C$ is an isometric embedding to get
		\begin{align*}
		d(T'c_i(y_1), T'c_i(y_2)) & = d(m_C \comp T'c_i(y_1), m_C \comp T'c_i(y_2)) \\
		& = d(Tc_i \comp m_{D_i}(y_1), Tc_i \comp m_{D_i}(y_2)) \\
		& = d(Tc_i(\ol{y}_1), Tc_i(\ol{y}_2)).
		\end{align*}
	The right-hand sides are also equal: for every $j \geq i$, since $m_{D_j}$ is an isometric embedding, we have
	\begin{align*}
	d(T'f_j(y_1), T'f_j(y_2)) & = d(m_{D_j} \comp T'f_j(y_1), m_{D_j} \comp T'f_j(y_2)) \\
	& = d(Tf_j \comp m_{D_i}(y_1), Tf_j \comp m_{D_i}(y_2)) \\
	& = d(Tf_j(\ol{y}_1), Tf_j(\ol{y}_2)).
	\end{align*}
	Thus, \eqref{eq:meet1} implies \eqref{eq:meet}, as desired.
	\end{description}
\end{enumerate}
\end{proof}

\section{From Finitary Varieties to Strongly Finitary Monads}

We present the concept of a variety of quantitative $\Sigma$-algebras, where $\Sigma$ is a finitary signature.
Assuming that strongly finitary endofunctors compose (see~\ref{OP})
we prove that every variety is isomorphic to $\Met^\Tmon$ for a strongly finitary monad $\Tmon$.
In the subsequent section we prove the converse: $\Met^\Tmon$ is isomorphic to a variety for every strongly finitary monad $\Tmon$.
For the base category $\UMet$ the same result holds -- but here no extra assumption is needed, see Remark~\ref{R:comp}.

\begin{notation}
\phantom{phantom}
\begin{enumerate}
	\item Throughout this section $\Sigma$ denotes a finitary signature, i.e., a set of (operation) symbols $\sigma$ with prescribed arities in $\Nat$.
	\item For convenience, we assume that a countable infinite set $V$ of variables has been chosen.
	\item Throughout the rest of our paper $\eps$ and $\delta$ denote non-negative rational numbers.
\end{enumerate}
\end{notation}

The following concept was introduced in~\cite{mpp17}.

\begin{definition}
A \emph{quantitative algebra} of signature $\Sigma$ is a metric space $A$ equipped with a nonexpanding operation $\sigma_A : A^n \to A$ for every $n$-ary symbol $\sigma \in \Sigma$.
Thus the distance between $\sigma_A(a_i)$ and $\sigma_A(b_i)$ is at most $\max \{ d(a_0,b_0), \dots, d(a_{n-1},b_{n-1}) \}$.
A \emph{homomorphism} to a quantitative algebra $B$ is a nonexpanding function $f : A \to B$ preserving the operations:
\[
f \comp \sigma_A = \sigma_B \comp f^n
\]
for every $\sigma$ of arity $n$ in $\Sigma$.

The category of quantitative algebras is denoted by $\Sigma\text{-}\Met$.
\end{definition}

Analogously, an \emph{ultra-quantitative algebra} is a quantitative algebra on an ultrametric space.
The corresponding full subcategory of $\Sigma\text{-}\Met$ is denoted by $\Sigma\text{-}\UMet$.

\begin{example}
\label{E:monoid}
Quantitative monoids.
These are monoids $M$ acting on a metric space with nonexpansive multiplicationfrom $M \times M$ to $M$.
Explicitly:
\begin{center}
if $d(x,x') = \eps_1$ and $d(y,y') = \eps_2$, then $d(xy,x'y') \leq \max\{ \eps_1, \eps_2 \}$.
\end{center}
A free quantitative monoid $TM$ on a space $M$ is the coproduct of powers $M^n = M \times \dots \times M$:
\[
TM = \Coprod_{n \in \Nat} M^n.
\]
This endofunctor is strongly finitary (Example~\ref{E:sa}~(2)).
\end{example}

\begin{remark}
In contrast, the category of monoids in the monoidal category $\Met$ does not yield a strongly finitary monad.
If
\[
M^{\underline{n}} = M \tensor \dots \tensor M
\]
denotes the $n$-th tensor power, a free monoid on a space $M$ is given by
\[
\wt{T}M = \Coprod_{n \in \Nat} M^{\underline{n}}.
\]
The corresponding endofunctor is not even enriched.
Consider $M = \{ a_1, a_2 \}$ with $d(a_1,a_2) = 1$.
We have the obvious morphisms $f_1, f_2 : \One \to M$ with $d(f_1,f_2) = 1$.
However, $d(f_1^{\underline{n}},f_2^{\underline{n}}) = n$, thus, $d(\wt{T}f_1,\wt{T}f_2) = \infty$.
This indicates that Mardare et al.\ made the right choice when defining quantitative algebras.
However, those are much more restrictive than their `relatives' using the tensor product.
For example both $(\mathbb{R},+,0)$ and $([0,1],\cdot,1)$ are monoids in $\Met$, but they are not quantitative monoids.
\end{remark}

\begin{example}
\label{E:term}
In universal algebra, for every set $X$ (of variables) the algebra $T_\Sigma X$ of \emph{terms} is defined as the smallest set containing the variables, and such that given $\sigma\in \Sigma$ of arity $n$ and $n$ terms $t_i$ ($i < n$) we obtain a (composite) term $\sigma(t_i)_{i<n}$.
For metric spaces this was generalized in~\cite{bmpp18} and we recall the construction here.
\begin{enumerate}
	\item Two terms $t$ and $t'$ in $T_\Sigma X$ are called \emph{similar} if they have the same structure and differ only in the names of variables.
	That is, similarity is the smallest equivalence relation on $T_\Sigma X$ such that
	\begin{enumerate}
		\item all variables in $X$ are pairwise similar, and
		\item a term $\sigma(t_i)_{i<n}$ is similar to precisely those terms $\sigma(t_i')_{i<n}$ such that $t_i$ is similar to $t_i'$ for each $i < n$.
	\end{enumerate}
	\item Given a metric space $M$ on a set $X$, we denote by $T_\Sigma M$ the quantitative algebra on $T_\Sigma X$ where the metric extends that of $M$ as follows: the distance of non-similar terms is $\infty$, and for similar terms $\sigma(t_i)_{i<n}$ and $\sigma(t_i')_{i<n}$ the distance is $\bigvee_{i<n} d(t_i, t_i')$.
	
	It follows that the metric space $T_\Sigma M$ is a coproduct of powers $M^k$ (with the supremum metric), one summand for each similarity class $[t]$ of terms $t$ on $k$ variables.
	The function $(t_i)_{i<n} \mapsto \sigma(t_i)_{i<n}$ of composite terms is clearly a nonexpanding operation $\sigma_{T_\Sigma X} : (T_\Sigma X)^n \to T_\Sigma X$.
	
\end{enumerate}
\end{example}

\begin{remark}
If $M$ is a (complete) ultrametric space, then so is $T_\Sigma M$: the categories $\UMet$ and $\CUMet$ are closed under finite products and coproducts in $\Met$.
\end{remark}

\begin{example}
For our discrete space $V = \{ x_i \mid i \in \Nat \}$ of variables we get the discrete algebra $T_\Sigma V$ of terms (in the usual sense of universal algebra).
\end{example}

\begin{remark}
\label{R:Barr}
\phantom{phantom}
\begin{enumerate}
	\item Barr proved in~\cite{B} that for every endofunctor $H$ of $\Met$ which generates a free monad $\Tmon$ the category $\Met^\Tmon$ is equivalent to the category of \emph{$H$-algebras}.
	These are pairs $(A,\alpha)$ consisting of a metric space $A$ and a morphism $\alpha : HA \to A$.
	The morphisms of $H$-algebras $f: (A,\alpha) \to (A',\alpha')$ are the nonexpanding maps $f: A \to A'$ satisfying $f \comp \alpha = \alpha' \comp Hf$.
	
	\item If $H$ is $\lambda$-accessible, then a free $H$-algebra exists on every metric space $X$.
	As proved in~\cite{A}, the free $H$-algebra is the colimit of the $\lambda$-chain in $\Met$ whose objects $W_i$ and connecting morphisms $w_{ij} : W_i \to W_j$ for $i \leq j < \lambda$ are given as follows:
	\begin{align*}
	W_0 & = \emptyset \\
	W_{i+1} & = HW_i + X \\
	W_j & = \colim_{i < j} W_i \text{ for limit ordinals } j < \lambda
	\end{align*}
	and
	\begin{align*}
	w_{i+1,j+1} & = H w_{i,j} + \id_X \text{ for all } i \leq j < \lambda \\
	(w_{ij})_{i < \lambda} & \text{ is the colimit cocone for limit ordinals } j.
	\end{align*}
	
	\item The free monad on a $\lambda$-accessible endofunctor $H$ is the monad of free $H$-algebras, see~\cite{B}.
	
	\item The category $\Sigma\text{-}\Met$ is clearly isomorphic to the category of $H_\Sigma$-algebras for the polynomial functor (Example~\ref{E:sa}).
\end{enumerate}
\end{remark}

\begin{proposition}
\label{P:poly}
For every metric space $X$ the quantitative algebra $T_\Sigma X$ is free on $X$ w.r.t.\ the inclusion map $\eta_X : X \to T_\Sigma X$.
\end{proposition}
\begin{proof}
The functor $H_\Sigma$ is clearly finitary.
Thus by Remark~\ref{R:Barr} a free $H_\Sigma$-algebra on $X$ is a colimit of the following $\omega$-chain $W$:
\[
W_0 = \emptyset, \qquad W_{i+1} = H_\Sigma W_i + X
\]
for all $i \in \Nat$.
It is easy to verify that each connecting morphism $w_{ij}$ is an isometric embedding.
Thus, the colimit is given by $\bigcup_{i < j} W_i$ with the induced metric.
It is also easy to verify that $W_i$ is isomorphic to the subspace of $T_\Sigma X$ on all terms of height less than $i$.
Thus, $T_\Sigma X = \bigcup_{i < \omega} W_i$ is the free $\Sigma$-algebra on $X$.
\end{proof}

\begin{notation}
\label{N:ext}
For every $\Sigma$-algebra $A$ and every morphism $f: X \to A$ we denote by $f^\sharp : T_\Sigma X \to A$ the unique homomorphic extension (defined in $t = \sigma(t_i)_{i<n}$ by $f^\sharp(t) = \sigma_A(f^\sharp(t_i)_{i<n})$).
\end{notation}

\begin{corollary}
\label{C:term}
The functor $T_\Sigma$ is a coproduct of functors $(\blank)^n$ for $n \in \Nat$. Thus the monad $\Tmon_\Sigma$ of free $\Sigma$-algebras is strongly finitary.
\end{corollary}

Indeed, for every space $M$ we have seen that $T_\Sigma M$ is a coproduct of the powers $M^n$, one summand for each similarity equivalence class of terms on $n$ variables.
Since $M$ itself forms just one class, we see that the similarity classes are independent of the choice of $M$.
Thus $T_\Sigma (\blank) = \coprod (\blank)^n$ where the coproduct ranges over similarity classes of $T_\Sigma 1$.
It follows that $T_\Sigma$ is strongly finitary by Example~\ref{E:sa}.

The following definition stems from~\cite{mpp16}.
Recall that $\eps$ denotes an arbitrary non-negative rational number.

\begin{remark}
If $M$ is an ultrametric space, then $T_\Sigma M$ is an ultra-quantitative algebra.
This is the free algebra of $\Sigma\text{-}\UMet$ on $M$; yielding a strongly finitary endofunctor $T_\Sigma$ on $\UMet$.
\end{remark}

\begin{definition}
	\label{D:qe}
	\phantom{phantom}
	\begin{enumerate}
		\item A \emph{quantitative equation} is an expression
		\[
		l =_{\eps} r
		\]
		where $l$ and $r$ are terms in $T_\Sigma V$.
		Ordinary equations $l = r$ are the special case of $\eps = 0$.
		\item A quantitative algebra $A$ \emph{satisfies} the quantitative equation if every interpretation of variables $f : V \to |A|$ fulfils
		\[
		d(f^\sharp(l), f^\sharp(r)) \leq \eps.
		\]
	\end{enumerate}
\end{definition}

\begin{example}
\phantom{phantom} \label{E:quant}
\begin{enumerate}
	\item Our restriction to \emph{rational} numbers $\eps$ plays no substantial role: consider $\eps > 0$ irrational.
	Then the `irrational quantitative equation` $l =_\eps r$ is equivalent to the set $l =_{\eps_n}$ of quantitative equations for an arbitrary decreasing sequence $(\eps_n)$ of rational numbers with $\lim \eps_n = \eps$ in $\mathbb{R}$.
	\item \emph{Almost commutative monoids}. For a given constant $\eps > 0$ we call a quantitative monoid (Example~\ref{E:monoid}) almost commutative if the distance of $ab$ and $ba$ is always at most $\eps$.
	This is presented by the following quantitative equation:
	\[
	x * y =_\eps y * x.
	\]
	\item \emph{Quantitative semilattices}. Recall that a semilattice (with $0$) is a poset with finite joins, or equivalently a commutative and idempotent monoid $(M,+,0)$. (The order $x \leq y$ is defined by $x + y = y$.)
	A quantitative semilattice is a quantitative monoid which is commutative and idempotent.
	\item \emph{Almost semilattices} are almost commutative quantitative monoids which are \emph{almost idempotent}, i.e.\ they satisfy
	\[
	x * x =_\eps x.
	\]
	\item \emph{Almost small spaces} are metric spaces with distances at most $\eps$ (a given constant).
	These are algebras with empty signature satisfying $x =_\eps y$.
\end{enumerate}
\end{example}

\begin{definition}
A \emph{variety} of quantitative (or ultra-quantitative) $\Sigma$-algebras is a full subcategory of $\Sigma\text{-}\Met$ (or $\UMet$, resp.) specified by a set $E$ of quantitative equations.
Thus an algebra lies in the variety iff it satisfies every equation of $E$.
\end{definition}

Observe that the category $\Sigma\text{-}\Met$ is complete: $\Met$ is complete (Proposition~\ref{P:coco}) and the forgetful functor to $\Met$ creates limits.
In particular, a product $\prod_{i \in I} A_i$ of $\Sigma$-algebras is the cartesian product with coordinate-wise operations and the supremum metric:
\[
d((a_i),(a_i')) = \sup_{i \in I} d(a_i,a_i').
\]
The same holds for $\Sigma\text{-}\UMet$.

\begin{lemma}
\label{L:fact}
The categories $\Sigma\text{-}\Met$ and $\Sigma\text{-}\UMet$ have the factorization system with
\begin{align*}
\EE & = \text{surjective homomorphisms, and} \\
\MM & = \text{homomorphisms that are isometric embeddings.}
\end{align*}
\end{lemma}

\begin{proof}
Given a homomorhism $f : X \to Y$ observe that $f[Y]$ is closed under the operations of $Y$.
The diagonall fill-in property lifts from $\Met$ (or $\UMet$) to $\Sigma\text{-}\Met$.
\end{proof}

\begin{definition}
Let $A$ be a quantitative $\Sigma$-algebra.
By a \emph{subalgebra} we mean a subobject in $\Sigma\text{-}\Met$ represented by an isometric embedding.
By a \emph{homomorphic image} we mean a quotient object in $\Sigma\text{-}\Met$ represented by a surjective homomorphism.

Analogously for subalgebras and homomorphic images in $\Sigma\text{-}\UMet$.
\end{definition}

The following was formulated in~\cite{mpp17}, a proof can be found in \cite{mu}, B19-B20:

\begin{theorem}[Birkhoff's Variety Theorem]
\label{T:BVT}
A full subcategory of $\Sigma\text{-}\Met$ or $\Sigma\text{-}\UMet$ is a variety iff it is closed under products, subalgebras, and homomorphic images.
\end{theorem}

In fact, loc.~cit.\ only works for $\Sigma\text{-}\Met$.
Since $\UMet$ is closed under products, subspaces and quotient spaces in $\Met$, that result immediately implies the case of $\Sigma\text{-}\UMet$, too.

\begin{corollary}
\label{C:left}
For every variety $\Vvar$ the forgetful functor $U: \Vvar \to \Met$ has a left adjoint $F \adj U$.
\end{corollary}

Indeed, since $\Vvar$ is closed under products and subalgebras (i.e., $\MM$-subobjects), it is a reflective subcategory of $\Sigma\text{-}\Met$ with reflections in $\EE$ (\cite{ahs}, Theorem~16.8).
Thus, a free algebra $FX$ of $\Vvar$ on a space $X$ is a reflection of the free algebra $T_\Sigma X$ of Proposition~\ref{P:poly}.

To see that the corresponding functor $F: \Met \to \Vvar$ is enriched, consider morphisms $f,g : X \to Y$ in $\Met$ with $d(f,g) = \eps$.
To verify $d(Ff, Fg) \leq \eps$, denote by $P \subseteq FX$ the set of all $p \in FX$ with $d(Ff(p), Fg(p)) \leq \eps$.
We prove that $P = FX$.
First, $P$ contains $\eta_X[X]$: if $p = \eta_X(x)$, then from $d(f(x),g(x)) \leq \eps$ we derive $d(Ff(p),Fg(p)) = d(\eta_Y \comp f(x), \eta_Y \comp g(x)) \leq \eps$ because $\eta_Y$ is nonexpanding.
Thus, to prove $P = FX$, it is sufficient to verify that $P$ is closed under the operations.
Let $\sigma$ be an $n$-ary symbol and $p = \sigma(p_i)$ with $p_i \in P$ for each $i < n$.
Since $Ff$ and $Fg$ are homomorphisms, we get
\begin{align*}
d(Ff(p), Fg(p)) & = d(\sigma_{FX}(Ff(p_i)), \sigma_{FX}(Fg(p_i))) \\
& \leq \sup_{i} d(Ff(p_i),Fg(p_i)) \leq \eps.
\end{align*}

\begin{notation}
For every variety $\Vvar$ we denote by $\Tmon_\Vvar$ the free-algebra monad on $\Met$, carried by $UF$.
\end{notation}

\begin{example}
\label{E:mons}
\phantom{phantom}
\begin{enumerate}
	\item The monad of quantitative monoids is the lifting of the word monad $TM = M^*$ on $\Set$ such that $M^* = \coprod_{n < \omega} M^n$ in $\Met$.
	(Words of different lengths have distance $\infty$.)
	\item The monad of commutative quantitative monoids is an analogous lifting of the corresponding monad on $\Set$: $TM$ is a coproduct of the spaces $M^n/{\sim}$ where $(x_i) \sim (y_i)$ iff there is a permutation $p$ on $n$ with $y_i = x_{p(i)}$ for all $i$.
	\item In contrast, the monad of almost commutative monoids is a lifting of the word monad : $TM = M^*$, but the metric is more complex than in (1).
	For example,
	\[
	d(x*y,y*x) = d(x,y) \wedge \eps.
	\]
	\item As proved in~\cite{mpp16}, the monad of quantitative semilattices is given by the finitary Hausdorff functor $\mathcal{H}_f$ (Example~\ref{E:Haus}~(2)).
\end{enumerate}
\end{example}

Recall the \emph{comparison functor}
\[
K: \Vvar \to \Met^{\Tmon_\Vvar}
\]
assigning to every algebra $A$ in $\Vvar$ the Eilenberg-Moore algebra
\[
KA \equiv T_\Vvar A \xrightarrow{\alpha} A
\]
given by the unique homomorphism $\alpha$ with $\alpha \comp \eta_A  = \id_A$.

\begin{example}
If $\Vvar = \Sigma\text{-}\Met$, then $\alpha = \eps_A^\sharp : T_\Vvar A \to A$ evaluates terms in $T_\Sigma A$ as elements of $A$.
The comparison functor is an isomorphism.
The verification is analogous to the classical $\Sigma$-algebras, see~\cite{maclane:cwm}, Theorem~VI.8.1.
\end{example}

\begin{definition}
\label{D:concr}
By a \emph{concrete category} over $\Met$ is meant a pair $(\A,U)$, where $\A$ is an (enriched) category and $U: \A \to \Met$ an (enriched) faithful functor: $d(f,g) = d(Uf,Ug)$ holds for all parallel pairs $f$, $g$ in $\A$.

A \emph{concrete functor} from $(\A,U)$ to $(\A',U')$ is a functor $F: \A \to \A'$ with $U = U' \comp F$.
If $F$ is an isomorphism, we say that $(\A,U)$ is \emph{concretely isomorphic} to $(\A',U')$.
\end{definition}

\begin{proposition}[See~\cite{rosicky:metric-monads}, 3.8]
\label{P:acc}
Every variety $\Vvar$ of algebras is concretely isomorphic to $\Met^{\Tmon_\Vvar}$, the Eilenberg-Moore algebras of its free-algebra monad: the comparison functor is a concrete isomorphism.
\end{proposition}

\begin{remark}
\label{R:dubuc}
In the next result we use a technique due to Dubuc~\cite{dubuc}.
This makes it possible to identify algebras for a monad $\Tmon$ with monad morphisms with domain $\Tmon$.
We recall this shortly here.
\begin{enumerate}
	\item Let $M$ be a metric space.
	We obtain the following adjoint situation:
	\[
	[\blank,M] \adj [\blank,M]: \Met \to \Met^\op.
	\]
	The corresponding monad on $\Met$ is enriched, we denote it by $\spitze{M,M}$.
	Thus
	\[
	\spitze{M,M}X = [[X,M],M]
	\]
	for all $X \in \Met$.
	\item Let $\Tmon$ be an enriched monad on $\Met$.
	Every algebra $\alpha: TM \to M$ in $\Met^\Tmon$ defines a natural transformation
	\[
	\wh{\alpha}: \Tmon \to \spitze{M,M}
	\]
	as follows.
	The component, for an arbitrary space $X$,
	\[
	\wh{\alpha}_X: TX \to [[X,M],M]
	\]
	is the adjoint transpose of the following composite
	\[
	TX \tensor [X,M] \xrightarrow{TX \tensor T(\blank)} TX \tensor [TX,TM] \xrightarrow{eval} TM \xrightarrow{\alpha} M.
	\]
	Moreover, $\wh{\alpha}$ is a monad morphism.
	
	Conversely, every monad morphism from $\Tmon$ to $\spitze{M,M}$ has the form $\wh{\alpha}$ for a unique algebra $(M,\alpha)$ in $\Met^\Tmon$.
	\item This bijection $(\wh{\blank})$ is coherent for monad morphisms as is shown in the following lemma.
\end{enumerate}
\end{remark}

\begin{lemma}
\label{L:amm}
Let $\gamma: \Smon \to \Tmon$ be a monad morphism.
Every algebra $\alpha: TA \to A$ in $\Met^\Tmon$ yields an algebra $\beta = \alpha \comp \gamma_A : SA \to A$ in $\Met^\Smon$, for which the following triangle commutes:
\[
\begin{tikzcd}
{\Smon} \arrow[rr, "\gamma"] \arrow[rd, swap, "\wh{\beta}"] & & {\Tmon} \arrow[ld, "\wh{\alpha}"] \\
& {\spitze{A,A} X} &
\end{tikzcd}
\]
\end{lemma}

\begin{proof}
Our task is to prove that the upper triangle in the following diagram commutes:
\[
\begin{tikzcd}
{\Smon X} \arrow[rr, "\gamma_X"] \arrow[rd, swap, "\wh{\beta}_X"] \arrow[rdd, bend right, swap, "\beta \comp Sf"] & & {\Tmon X} \arrow[ld, "\wh{\alpha}_X"] \arrow[ldd, bend left, "\alpha \comp Tf"] \\
& {\spitze{A,A} X} \arrow[d, "\pi_f"] & \\
& {A} &
\end{tikzcd}
\]
For every $f \in [X,A]$ denote by $\pi_f : \spitze{A,A}X = [[X,A],A] \to A$ the corresponding morphism.
Then all $\pi_f$ are clearly collectively monic.
Thus, it is sufficient to prove that all $\pi_f$ (for $f: X \to A$) merge the upper triangle.
Indeed, we use the definition of $\spitze{A,A} X$:
\[
\beta \comp SF = \alpha \comp \gamma_A \comp Sf = \alpha \comp Tf \comp \gamma_X.
\]
\end{proof}

\begin{notation}
\label{N:colim}
\phantom{phantom}
\begin{enumerate}
	\item For our standard set $V$ of variables $x_i$ ($i \in \Nat$) we put
	\[
	V_n = \{ x_i \mid i < n \}.
	\]
	Every term $t \in T_\Sigma V$ lies in $T_\Sigma V_n$ for some $n \in \Nat$.
	\item The endofunctor $(\blank)^n$ is the polynomial functor $H_{[n]}$ of the signature $[n]$ of one $n$-ary symbol $\delta$ (Remark~\ref{R:Barr}).
	Every term $t \in T_\Sigma V_n$ yields, by Yoneda lemma, a natural transformation from $H_{[n]}$ to $T_\Sigma$, and we obtain the corresponding monad morphism
	\[
	\ol{t}: \Tmon_{[n]} \to \Tmon_\Sigma.
	\]
	It substitutes in every term of $T_{[n]} X$ all occurences of $\delta$ by $\sigma$. More precisely, the component $\ol{t}_X$ is identity on variables from $X$, and for a term $s = \delta(s_i)_{i<n}$ it is recursively given by
	\[
	\ol{t}_X(s) = t(\ol{t}_X(s_i))_{i<n}.
	\]
	\item Every $n$-ary symbol $\sigma \in \Sigma$ is identified with the term $\sigma(x_i)_{i<n}$ in $T_\Sigma V_n$.
	\item For monads $\Tmon$ and $\Smon$ we denote by
	\[
	\llbracket \Tmon, \Smon \rrbracket
	\]
	the metric space $\Mon(\Met)(\Tmon,\Smon)$.
	That is, the distance of monad morphisms $\phi,\psi: \Tmon \to \Smon$ is $\sup_{X \in \Met} d(\phi_X,\psi_X)$.
\end{enumerate}
\end{notation}

\begin{lemma}
\label{L:free}
Let $l, r$ be terms in $T_\Sigma V_n$.
A quantitative algebra expressed by $\alpha: T_\Sigma A \to A$ in $\Met^{\Tmon_\Sigma}$ satisfies the quantitative equation $l =_\eps r$ iff the composite monad morphisms $\wh{\alpha} \comp \ol{l}$ and $\wh{\alpha} \comp \ol{r}$:
\[
\begin{tikzcd}
{\Tmon_{\Delta_n}} \arrow[r, bend left, "\ol{l}"] \arrow[r, swap, bend right, "\ol{r}"] & {\Tmon_\Sigma} \arrow[r, "\wh{\alpha}"] & {\spitze{A,A}}
\end{tikzcd}
\]
have distance at most $\eps$ in $ \llbracket \Tmon_{\Delta_n}, \spitze{A,A} \rrbracket $.
\end{lemma}

\begin{proof}
\phantom{phantom}
\begin{enumerate}
	\item Assuming $d(\wh{\alpha} \comp \ol{l}, \wh{\alpha} \comp \ol{r}) \leq \eps$, we consider the $V$-components: they have distance at most $\eps$ in $ \llbracket T_\Delta,[[V,A],A] \rrbracket $.
	Applied to $\delta(x_i) \in T_\Delta V$, they yield elements of $[[V,A],A]$ of distance at most $\eps$.
	Recall that
	\[
	\wh{\alpha} \comp \ol{l}_V (\delta(x_i)) = \wh{\alpha}(l(x_i))
	\]
	assigns to $f \in [V,A]$ the value $f^\sharp(l(x_i))$.
	Analogously for $\wh{\alpha} \comp \ol{r}_V$.
	We thus see that
	\[
	d(f^\sharp(l),f^\sharp(r)) \leq \eps
	\]
	for all $f \in [V,A]$.
	Therefore $A$ satisfies $l =_\eps r$.
	\item Assuming that $A$ satisfies $l =_\eps r$, we prove that for every space $X$ we have
	\[
	d(\wh{\alpha}_X \comp \ol{l}_X, \wh{\alpha}_X \comp \ol{r}_X) \leq \eps.
	\]
	Our assumption is that every evaluation $f : V \to |A|$ fulfils
	\[
	d(f^\sharp(l),f^\sharp(r)) \leq \eps.
	\]
	Our task is to prove that for every term $t \in T_\Delta V$ we have
	\[
	d(\wh{\alpha}_X \comp \ol{l}_X (t), \wh{\alpha}_X \comp \ol{r}_X (t)) \leq \eps.
	\]
	Equivalently, given $f$, we have
	\begin{equation*}
	d(f^\sharp(\ol{l}_X(t)),f^\sharp(\ol{r}_X(t))) \leq \eps. \tag{\textasteriskcentered} \label{eq:proof}
	\end{equation*}
	We proceed by induction on the height $h(t)$ of $t$ (where variables have height $0$ and the height of a term $t = \sigma(t_i)_{i < n}$ is $h(t) = 1 + \max_{i < n} h(t_i)$).
	If the height is $0$, then $\ol{l}_X(t) = \ol{r}_X(t) = t$.
	Suppose that $t = \sigma(t_i)_{i<n}$ and for each $i$ we have already proved
	\begin{equation*}
	d(\wh{\alpha}_X \comp \ol{l}_X (t_i), \wh{\alpha}_X \comp \ol{r}_X (t_i)) \leq \eps. \tag{\textasteriskcentered\textasteriskcentered} \label{eq:assume}
	\end{equation*}
	We have $\ol{l}_X(t) = \sigma(\ol{l}_X(t_i))$, and therefore
	\[
	f^\sharp(\ol{l}_X(t)) = \sigma_A(f^\sharp(\ol{l}_X(t_i)).
	\]
	Analogously for $f^\sharp(\ol{r}_X(t))$.
	Since $\sigma_A$ is nonexpanding, the inequalities~\eqref{eq:assume} imply \eqref{eq:proof}, as desired.
\end{enumerate}
\end{proof}

\begin{construction}
\label{C:diagram}
For every variety $\Vvar$ we define a weighted diagram in $\Mon_\sf(\Met)$ (Notation~\ref{N:fsf}), and prove below that the free-algebra monad of $\Vvar$ is its colimit.
Suppose that $\Vvar$ is specified by a collection of quantitative equations as follows:
\[
l_i =_{\eps_i} r_i \text{ with } l_i,r_i \in T_\Sigma V_{n(i)}
\]
for $i \in I$.
The domain $\D$ of our diagram is the discrete category $I$ extended by an object $a$ and two cocones
\[
\lambda_i,\rho_i: i \to a
\]
Using Notation~\ref{N:colim} (2), the functor $D: \D \to \Mon_\sf (\Met)$ is given by
\[
Di = \Tmon_{[n(i)]} \text{ and } Da = \Tmon_\Sigma, 
\]
where for morphisms we put
\[
D\lambda_i = \ol{l}_i \text{ and } D\rho_i = \ol{r}_i \text{ ($i \in I$).}
\]
The weight $W: \D^\op \to \Met$ is defined by
\begin{align*}
W a & = \{ 0 \} \\
W i &= \{ l,r \} \text{ with } d(l,r) = \eps_i
\end{align*}
and on morphisms
\[
W \lambda_i : 0 \mapsto l \text{ and } W \rho_i : 0 \mapsto r.
\]
\end{construction}

\begin{theorem}
\label{T:var-mon}
Let $\Vvar$ be a variety of quantitative algebras.
Assuming that strongly finitary endofunctors on $\Met$ are closed under composition, the monad $\Tmon_\Vvar$ is the colimit of the above diagram in $\Mon_\sf(\Met)$.
\end{theorem}

\begin{proof}
Denote the colimit (which exists by~Proposition~\ref{P:cls}) by $\Tmon$.
We then have $T = \Colim{W}{D}$ both in $\Mon_\sf(\Met)$ and in $\Mon_\fin(\Met)$: the former category is by Proposition~\ref{P:cls} closed under colimits in the latter one.
We prove that $\Tmon$ is the free-algebra monad of $\Vvar$.
\begin{enumerate}
	\item We have an isomorphism
	\[
	\psi_\Smon : \llbracket \Tmon, \Smon \rrbracket \to [\D^\op,\Met](W, \llbracket D \blank, \Smon \rrbracket ).
	\]
	natural in $\Smon \in \Mon_\fin(\Met)$.
	In particular, this yields the unit
	\[
	\nu = \psi_\Tmon(\id_\Tmon) : W \to \llbracket D \blank, \Tmon \rrbracket.
	\]
	Its $a$-component assigns to $0$ a monad morphism
	\[
	\gamma = \nu_a (0) : \Tmon_\Sigma \to \Tmon.
	\]
	Its $i$-component is, due to the naturality of $\nu$, given by
	\[
	\nu_i : l \mapsto \gamma \comp \ol{l}_i \text{ and } r \mapsto \gamma \comp \ol{r}_i.
	\]
	Since this component $\nu_i$ is nonexpanding and $d(l,r) = \eps_i$ in $W i$, we conclude
	\begin{equation}
	\label{eq:form1}
	\tag{\textasteriskcentered}
	d(\gamma \comp \ol{\lambda}_i, \gamma \comp \ol{\rho}_i) \leq \eps_i \text{ for } i \in I.
	\end{equation}

	\item For every monad morphism $\phi : \Tmon \to \Smon$ in $\Mon_\fin(\Met)$ the natural transformation $\psi_\Smon(\phi): W \to \llbracket D\blank, \Smon \rrbracket$ has the following components:
	\[
	\psi_\Smon(\phi)_a : 0 \mapsto (\Tmon_\Sigma \xrightarrow{\gamma} \Tmon \xrightarrow{\phi} \Smon)
	\]
	and
	\begin{align*}
	\psi_\Smon(\phi)_i : l & \mapsto (\Tmon_\Delta \xrightarrow{\ol{l}_i} \Tmon_\Sigma \xrightarrow{\gamma} \Tmon \xrightarrow{\phi} \Smon) \\
	r & \mapsto (\Tmon_\Delta \xrightarrow{\ol{r}_i} \Tmon_\Sigma \xrightarrow{\gamma} \Tmon \xrightarrow{\phi} \Smon).
	\end{align*}
	This follows from the naturality of $\psi_\Smon$.
	
	\item We prove that the components of $\gamma$ are surjective.
	Indeed, $\Tmon_\Sigma$ preserves surjective morphisms (Example~\ref{E:sa}~(2) and Example~\ref{E:term}~(2)). Thus by Proposition~\ref{P:fact} we have a factorization in $\Mon_\fin(\Met)$ of $\gamma$ as follows
	\[
	\gamma = \Tmon_\Sigma \xrightarrow{e} \Smon \xrightarrow{m} \Tmon
	\]
	where $e$ has surjective components and those of $m$ are isometries.
	We will prove that $m$ is invertible.
	
	From the inequalities~\eqref{eq:form1} we derive, since $m$ has isometric components, that
	\[
	d(e \comp \ol{\lambda}_i, e \comp \ol{\rho}_i) \leq \eps_i \text{ for } i \in I.
	\]
	Consequently, we can define a natural transformation $\delta : W \to \llbracket D \blank, \Smon \rrbracket $ by the following components
	\begin{align*}
	\delta_a : 0 & \mapsto e, \\
	\delta_i : l & \mapsto e \comp \ol{\lambda}_i, \\
			   r & \mapsto e \comp \ol{\rho}_i \quad (i \in I).
	\end{align*}
	Since $\psi_\Smon$ is invertible, we obtain a monad morphism
	\[
	\ol{m} : \Tmon \to \Smon \text{ with } \psi_\Smon(\ol{m}) = \delta.
	\]
	Apply the naturality square
	\[
	\begin{tikzcd}
	{ \llbracket \Tmon, \Smon \rrbracket }
	\arrow[r, "\psi_\Smon"]
	\arrow[d, swap, "m \comp (\blank)"]
	& { \llbracket D \blank, \Smon \rrbracket }
	\arrow[d, "m \comp (\blank)"]\\
	{ \llbracket \Tmon, \Tmon \rrbracket }
	\arrow[r, swap, "\psi_\Tmon"]
	& { \llbracket D \blank, \Tmon \rrbracket }
	\end{tikzcd}
	\]
	to $\ol{m}$: the lower passage yields $\psi_\Tmon(m \comp \ol{m})$ and the upper one yields $m \comp \delta$.
	The equality of the $a$-components thus implies
	\[
	\psi_\Tmon(m \comp \ol{m})_a = m \comp \delta_a = m \comp e = \gamma.
	\]
	Since $\gamma = \psi_\Tmon(\id_\Tmon)$ and $\psi_\Tmon$ is invertible, we have proved $m \comp \ol{m} = \id_\Tmon$.
	Thus $m$ is monic and split epic, proving that it is invertible.
	
	\item Every algebra $\alpha : T_\Sigma A \to A$ in $\Vvar$ defines a natural transformation
	\[
	\delta^\alpha : W \to \llbracket D \blank, \spitze{A,A} \rrbracket
	\]
	with components
	\[
	\delta^\alpha_a : 0 \mapsto \wh{\alpha} : \Tmon_\Sigma \to \spitze{A,A}
	\]
	(Remark~\ref{R:dubuc}) and
	\begin{align*}
	\delta^\alpha_a: l & \mapsto \wh{\alpha} \comp \ol{\lambda}_i, \\
					 r & \mapsto \wh{\alpha} \comp \ol{\rho}_i.
	\end{align*}
	Indeed, $\delta^\alpha_i$ is nonexpanding because $A$ satisfies $d(\wh{\alpha} \comp \ol{\lambda}_i, \wh{\alpha} \comp \ol{\rho}_i) \leq \eps_i$ by Lemma~\ref{L:free}.
	The naturality of $\delta^\alpha$ is obvious.
	
	\item We define a concrete isomorphism
	\[
	E : \Met^\Tmon \to \Vvar
	\]
	as follows.
	On objects $\alpha: TA \to A$ we use that $\gamma : \Tmon_\Sigma \to \Tmon$ is a monad morphism to obtain a $\Sigma$-algebra $\alpha' = \alpha \comp \gamma_A : T_\Sigma A \to A$.
	This algebra lies in $\Vvar$ because $\wh{\alpha}' = \wh{\alpha} \comp \gamma$ (Remark~\ref{R:dubuc}) and $\wh{\alpha}$ is nonexpanding, thus~\eqref{eq:form1} yields
	\[
	d(\wh{\alpha}' \comp \ol{\lambda}_i, \wh{\alpha}' \comp \ol{\rho}_i) \leq \eps_i \text{ for } i \in I.
	\]
	Now apply Lemma~\ref{L:free}.
	Thus we can define $E$ on objects by
	\[
	E(A,\alpha) = (A, \alpha').
	\]
	
	We next prove that given two algebras $(A,\alpha)$ and $(B,\beta)$ in $\Met^\Tmon$, then a nonexpanding map $f : A \to B$ is a homomorphism in $\Met^\Tmon$ iff it is a $\Sigma$-homomorphism.
	This follows from the diagram below:
	\[
	\begin{tikzcd}
	{T_\Sigma A}
	\arrow[rr, "\alpha'"]
	\arrow[rd, "\gamma_A"]
	\arrow[ddd, swap, "T_\Sigma f"]
	 & & A
	\arrow[ddd, "f"]
	 \\
	& TA
	\arrow[ur, "\alpha"]
	\arrow[d, swap, "Tf"]
	& \\
	& TB
	\arrow[rd, "\beta"]
	& \\
	{T_\Sigma B}
	\arrow[ur, "\gamma_B"]
	\arrow[rr, swap, "\beta'"]
	& & B
	\end{tikzcd}
	\]
	The left-hand part is the naturality of $\gamma$, and both triangles clearly commute.
	Since by (3) we know that $\gamma_A$ is epic, we conclude that the outward square commutes iff the right-hand part does.
	
	We conclude that by putting $Ef = f$ we obtain a concrete, fully faithful functor.
	It is invertible because given an algebra $\alpha: T_\Sigma A \to A$ in $\Vvar$ we apply the isomorphism $\psi_{\spitze{A,A}}$ to $\delta^\alpha : W \to \llbracket D \blank, \spitze{A,A} \rrbracket $ to obtain a unique $\alpha_0 : TA \to A$ in $\Met^\Tmon$ with $\psi_{\spitze{A,A}} (\wh{\alpha}_0) = \delta^\alpha$.
	By (2) the $a$-component of $\psi_{\spitze{A,A}}(\wh{\alpha}_0)$ is given by $0 \mapsto \wh{\alpha}_0 \comp \gamma$. Since $\delta^\alpha_a (0) = \wh{\alpha}$, this proves $\wh{\alpha}_0 \comp \gamma = \wh{\alpha}$.
	Lemma~\ref{L:amm} yields $\alpha_0 \comp \gamma_A = \alpha$.
	Thus $E^{-1}$ is the concrete functor defined by $E^{-1} (A,\alpha) = (A, \alpha_0)$.
	
	\item Since $\Vvar$ is concretely isomorphic to $\Met^{\Tmon_\Vvar}$, where $\Tmon_\Vvar$ is the free-algebra monad (Proposition~\ref{P:acc}), we conclude that the categories $\Met^\Tmon$ and $\Met^{\Tmon_\Vvar}$ are concretely isomorphic.
	This proves that the monads $\Tmon$ and $\Tmon_\Vvar$ are isomorphic (\cite{barr+wells:toposes}, Theorem~3.6.3).
	Therefore $\Tmon$ is a free-algebra monad for $\Vvar$ as claimed.  
\end{enumerate}
\end{proof}

\begin{corollary}
\label{C:sf}
Assume that strongly finitary endofunctors on $\Met$ compose. Then the free-algebra monad for a finitary variety of quantitative algebras is strongly finitary.
\end{corollary}

\begin{remark}
For the base category $\UMet$ we proceed completely analogously:
\begin{enumerate}
	\item For every variety $\Vvar$ of ultra-quantitative algebras a free-algebra monad $\Tmon_\Vvar$ on $\UMet$ exists, and $\Vvar$ is concretely isomorphic to $\UMet^{\Tmon_\Vvar}$.
	This is analogous to Corollary~\ref{C:left} and Proposition~\ref{P:acc}.
	\item For a monad $\Tmon$ on $\UMet$ and an ultrametric space $M$ we have a bijection between algebras for $\Tmon$ on $M$ and monad morphisms from $\Tmon$ to $\spitze{M,M}$ as in Remark~\ref{R:dubuc}.
	\item An ultra-quantitative algebra expressed by $\alpha: T_\Sigma M \to M$ in $\UMet^{\Tmon_\Vvar}$ satisfies $l =_\eps r$ iff $d(\wh{\alpha} \comp l, \wh{\alpha} \comp r) \leq \eps$.
	This is proved as Lemma~\ref{L:free}.
	\item For the weighted diagram in $\Mon_\sf(\UMet)$ constructed as in Construction~\ref{C:diagram} the monad $\Tmon_\Vvar$ is its colimit in $\Mon_\sf(\UMet)$.
	This is proved as Theorem~\ref{T:var-mon}, but here no extra assumption is needed since strongly finitary endofunctors on $\UMet$ compose (Remark~\ref{R:comp}).
	We obtain the following result:
\end{enumerate}
\end{remark}

\begin{corollary}
\label{C:cor-ultra}
The free-algebra monad for a finitary variety of ultra-quantitative algebras is strongly finitary.
\end{corollary}

\section{From Strongly Finitary Monads to Varieties}

Throughout this section $\Tmon = (T,\mu,\eta)$ denotes a strongly finitary monad on $\Met$ or $\UMet$.
We construct a finitary variety of quantitative (or ultra-quantitative) algebras with $\Tmon$ as its free-algebra monad.
This leads to the main result: a bijection between finitary varieties and strongly finitary monads.

\begin{remark}
\phantom{phantom}
\label{R:star}
\begin{enumerate}
	\item Recall that $(TX, \mu_X)$ is the free algebra of $\Met^\Tmon$ on $\eta_X : X \to TX$: for every algebra $(A,\alpha)$ and every morphism $f: X \to A$ the morphism
	\[
	f^* = \alpha \comp Tf : TX \to A
	\]
	is the unique homomorphism with $f = f^* \comp \eta_X$.
	\item Given $(A,\alpha) = (TX, \mu_X)$, we have $\eta_X^* = \id_{TX}$.
	And for morphisms $f: X \to TY$ and $g: Y \to TZ$ we have $g^* \comp f^* = (g^* \comp f)^* : TX \to TZ$.
	Indeed, $g^* \comp f^*$ is a homomorphism extending $g^* \comp f$: we have $(g^* \comp f^*) \comp \eta_X = g^* \comp f$.
\end{enumerate}
\end{remark}

\begin{notation}
For the given set $V = \{ x_k \mid k \in \Nat \}$ of variables we denote by $V_n$ the set (discrete space) $\{ x_k \mid k \leq n \}$.
\end{notation}

\begin{construction}
\label{C:eq}
For a strongly finitary monad $\Tmon$ on $\Met$ or $\UMet$ we define a signature $\Sigma$ by
\[
\Sigma_n = |TV_n| \text{ for all } n \in \Nat.
\]
Recall that an $n$-ary operation symbol $\sigma$ is viewed as a term $\sigma(x_i)_{i<n}$ in $T_\Sigma V_n$.
Example: every variable $x_i \in V_n$ yields an $n$-ary symbol $\eta_{V_n}(x_i)$ which is a term.

Our variety $\Vvar_\Tmon$ is presented by three types of equations, where $n$ and $m$ denote arbitrary natural numbers:
\begin{enumerate}
	\item $l =_\eps r$ for all $l,r \in TV_n$ with $d(l,r) \leq \eps$ in $TV_n$.
	\item $k^*(\sigma) = \sigma(k(x_i))_{i<n}$ for all $\sigma \in TV_n$ and all maps $k : V_n \to |TV_m|$.
	\item $\eta_{V_n}(x_i) = x_i$ for all $i < n$.
\end{enumerate}
\end{construction}

\begin{lemma}
Every algebra $\alpha: TA \to A$ in $\Met^\Tmon$ or $\UMet^\Tmon$ defines a $\Sigma$-algebra $A$ in $\Vvar_\Tmon$ with operations $\sigma_A : A^n \to A$ for $\sigma \in TV_n$ given by
\[
\sigma_A(a_i)_{i<n} = a^*(\sigma) \text{ for } a : V_n \to A, x_i \mapsto a_i.
\]
\end{lemma}
\begin{proof}
\phantom{phantom}
\begin{enumerate}[label=(\alph*)]
	\item The mapping $\sigma_A$ is nonexpanding: given $d((a_i)_{i<\omega},(b_i)_{i<\omega}) = \eps$ in $A^\omega$, the corresponding maps $a, b : V_n \to A$ fulfil $d(a,b) = \eps$.
	Since $T$ is enriched, this yields $d(Ta,Tb) \leq \eps$.
	Finally $\alpha$ is nonexpanding and $a^* = \alpha \comp Ta$, $b^* = \alpha \comp Tb$, thus $d(a^*,b^*) \leq \eps$.
	In particular $d(a^*(\sigma),b^*(\sigma)) \leq \eps$.
	\item The quantitative equations (1)-(3) hold:
	\begin{description}
		\item[\textnormal{Ad (1)}] Given $l,r \in TV_n$ with $d(l,r) \leq \eps$, then for every map $a : V_n \to A$ we have $d(a^*(l),a^*(r)) \leq \eps$.
		Thus $d(l_A(a_i),r_A(a_i)) \leq \eps$ for all $(a_i) \in A^\omega$.
		\item[\textnormal{Ad (2)}] Given $a : V_n \to A$ we prove $(k^*(\sigma))_A(a_j) = \sigma_A(k(x_i))(a_j)$.
		The left-hand side is $a^*(k^*(\sigma)) = (a^*k)^*(\sigma)$ by Remark~\ref{R:star}~(2).
		The right-hand one is $a^*(\sigma_A(k(x_i))) = (a^*k)^*(\sigma)$, too.
		\item[\textnormal{Ad (3)}] Recall that $\alpha \comp \eta_A = \id$ and $Ta \comp \eta_{V_n} = \eta_A\comp a$ for every map $a : V_n \to A$.
		Therefore
		\begin{align*}
		(\eta_{V_n}(x_i))_A(a_j) & = a^*(\eta_{V_n}(x_i)) \\
		& = \alpha \comp Ta \comp \eta_{V_n} (x_i) \\
		& = a(x_i) = a_i. 
		\end{align*}
		
	\end{description}
\end{enumerate}
\end{proof}

\begin{remark}
\label{R:hom}
Moreover, every homomorphism $h : (A,\alpha) \to (B,\beta)$ in $\Met^\Tmon$ or $\UMet^\Tmon$ is also a homomorphism between the corresponding $\Sigma$-algebras.
Indeed, given $\sigma \in TV_n$ we verify $h \comp \sigma_A = \sigma_B\comp h^n : A^n \to B$.
For every map $a : V_n \to A$ we want to prove $h(a^*(\sigma)) = b^*(\sigma)$ where $b : V_n \to B$ is given by $x_i \mapsto h(a_i)$, that is, $b = h \comp a$.
This follows from $h \comp a^* = (h \comp a)^*$: indeed, we know that $h \comp \alpha = \beta \comp Th$, thus
\[
h \comp a^* = h \comp \alpha \comp Ta = \beta \comp Th \comp Ta = \beta \comp T(h \comp a) = (h \comp a)^*.
\]
\end{remark}

\begin{theorem}
\label{T:sfv}
Every strongly finitary monad $\Tmon$ on $\Met$ or $\UMet$ is the free-algebra monad of the variety $\Vvar_\Tmon$.
\end{theorem}
\begin{proof}
We present a proof for $\Met$, the case $\UMet$ is completely analogous.
For every metric space $M$ we want to prove that the $\Sigma$-algebra associated with $(TM,\mu_M)$ is free in $\Vvar_\Tmon$ w.r.t.\ the universal map $\eta_M$.
Then the theorem follows from Proposition~\ref{P:acc}.

We have two strongly finitary monads, $\Tmon$ and the free algebra monad of $\Vvar_\Tmon$ (Corollary~\ref{C:sf}).
Thus, it is sufficient to prove the above for finite discrete spaces $M$.
Then this extends to all finite spaces because we have $M = \Colim{W_0}{D_M}$ (Proposition~\ref{P:prec}) and both monads preserve this colimit.
Since they coincide on all finite discrete spaces, they coincide on $M$.
Finally, the above extends to all spaces $M$: we have a directed colimit $M = \colim_{i \in I} M_i$ of the diagram of all finite subspaces $M_i$ ($i \in I$) which both monads preserve.

Given a finite discrete space $M$, we can assume without loss of generality $M = V_n$ for some $n \in \Nat$.
For every algebra $A$ in $\Vvar_\Tmon$ and a map $f : V_n \to A$, we prove that there exists a unique $\Sigma$-homomorphism $\ol{f} : TV_n \to A$ with $f = \ol{f} \comp \eta_{V_n}$.
	\begin{description}
		\item[Existence] Define $\ol{f}(\sigma) = \sigma_A(f(x_i))_{i<n}$ for every $\sigma \in TV_n$.
		The equality $f = \ol{f} \comp \eta_{V_n}$ follows since $A$ satisfies the equations $\eta_{V_n}(x_i) = x_i$, thus the operation of $A$ corresponding to $\eta_{V_n}(x_i)$ is the $i$-th projection.
		The map $\ol{f}$ is nonexpanding: given $d(l,r) \leq \eps$, the algebra $A$ satisfies $l =_\eps r$.
		Therefore for every $n$-tuple $f : V_n \to A$ we have
		\[
		d(l_A(f(x_i)),r_A(f(x_i))) \leq \eps.
		\]
		To prove that $\ol{f}$ is a $\Sigma$-homomorphism, take an $m$-ary operation symbol $\tau \in TV_m$.
		We prove $\ol{f} \comp \tau_{V_m} = \tau_A \comp \ol{f}^m$.
		This means that every $k : V_m \to TV_n$ fulfils
		\[
		\ol{f} \comp \tau_{V_m} (k(x_j))_{j < m} = \tau_A \comp \ol{f}^m (k(x_j))_{j < m}.
		\]
		The definition of $\ol{f}$ yields that the right-hand side is $\tau_A(k(x_j)_A(f(x_i)))$.
		Due to equation (2) in Construction~\ref{C:eq} with $\tau$ in place of $\sigma$ this is $k^*(\tau)_A(f(x_i))$.
		The left-hand side yields the same result since
		\[
		\ol{f}^m(k(x_j)) = (k(x_j))_A (f(x_i)).
		\]
		\item[Uniqueness] Let $\ol{f}$ be a nonexpanding $\Sigma$-homomorphism with $f = \ol{f} \comp \eta_{V_n}$.
		In $TV_n$ the operation $\sigma$ asigns to $\eta_{V_n}(x_i)$ the value $\sigma$.
		(Indeed, for every $a : n \to |TV_n|$ we have $\sigma_{TV_n}(a_i) = a^*(\sigma) = \mu_{V_n} \comp Ta(\sigma)$.
		Thus $\sigma_{TV_n}(\eta_{V_n}(x_i)) = \mu_{V_n} \comp T \eta_{V_n} (\sigma) = \sigma^n$.)
		Since $\ol{f}$ is a homomorphism, we conclude
		\[
		f(\sigma) = \sigma_A(\ol{f} \comp \eta_{V_n} (x_i)) = \sigma_A(f(x_i))
		\]
		which is the above formula.
	\end{description}

\end{proof}

Recall the concept of concrete functor over $\Met$ from Definition~\ref{D:concr}.

\begin{notation}
\phantom{phantom}
\begin{enumerate}
	\item The category of finitary varieties of quantitative algebras and concrete functors is denoted by
	\[
	\Var_\fp(\Met).
	\]
	Analogously $\Var_\fp(\UMet)$ is the category of varieties of ultra-quantitative algebras.
	\item Recall that the category of strongly finitary monads and monad morphisms is denoted by
	\[
	\Mon_\sf(\UMet) \text{ and } \Mon_\sf(\Met).
	\]
\end{enumerate}
\end{notation}

\begin{theorem}[Main theorem]
\label{T:main}
The category $\Var_\fp(\UMet)$ of finitary varieties of ultra-quantitative algebras is dually equivalent to the category $\Mon_\sf(\UMet)$ of strongly finitary monads on $\UMet$.
\end{theorem}

\begin{proof}
\phantom{phantom}
\begin{enumerate}
	\item Given monads $\Tmon$ and $\Smon$ on $\UMet$ we recall the bijection between monad morphisms and concrete functors, see e.g.~\cite{barr+wells:toposes}, Theorem~3.3.
	It assigns to every monad morphism $\phi: \Smon \to \Tmon$ the concrete functor $\ol{\phi} : \UMet^\Tmon \to \UMet^\Smon$ given by
	\[
	TA \xrightarrow{\alpha} A \; \; \mapsto \; \; SA \xrightarrow{\phi_A} TA \xrightarrow{\alpha} A.
	\]
	The map $\phi \mapsto \ol{\phi}$ is bijective and preserves composition.
	
	We denote by $(\blank)^*$ the inverse, assigning to a concrete functor $F : \UMet^\Tmon \to \UMet^\Smon$ the corresponding monad morphism $F^* : \Smon \to \Tmon$.
	This also preserves composition.
	\item For every variety $\Vvar$ the comparison functor $K_\Vvar : \Vvar \to \UMet^{\Tmon_\Vvar}$ is invertible (Proposition~\ref{P:acc}).
	Using Corollary~\ref{C:sf}, we define a functor
	\[
	\Phi: \Var_\fp(\UMet)^\op \to \Mon_\sf(\UMet)
	\]
	on objects $\Vvar$ by
	\[
	\Phi(\Vvar) = \Tmon_\Vvar.
	\]
	On morphisms $F : \Vvar \to \Wvar$ we define $\Phi(F)$ by forming the concrete functor
	\[
	\UMet^{\Tmon_\Vvar} \xrightarrow{K_\Vvar^{-1}} \Vvar \xrightarrow{F} \Wvar \xrightarrow{K_\Wvar} \UMet^{\Tmon_\Wvar}
	\]
	and putting
	\[
	\Phi(F) = (K_\Wvar F K_\Vvar^{-1})^* : \Tmon_\Wvar \to \Tmon_\Vvar.
	\]
	Since $(\blank)^*$ preserves composition, so does $\Phi$.
	\item $\Phi$ is obviously faithful and, since $(\blank)^*$ is a bijection, full.
	It is an equivalence functor because every monad $\Tmon$ in $\Mon_\sf(\UMet)$ is isomorphic to $\Phi(\Vvar_\Tmon)$ by Theorem~\ref{T:sfv}.
\end{enumerate}
\end{proof}

For varieties of quantitative algebras the procedure is the same, but here we use Corollary~\ref{C:sf}, so that we need an extra assumption:

\begin{theorem}
\label{T:mainmet}
Assume that strongly finitary endofunctors on $\Met$ compose.
Then the category $\Var_\fp(\Met)$ of finitary varieties of quantitative algebras is dually equivalent to the category $\Mon_\sf(\Met)$ of strongly finitary monads on $\Met$
\end{theorem}

\section{Monads on Complete Ultrametric Spaces}

For the category $\CUMet$ of (extended) complete ultrametric spaces (a full subcategory of $\Met$) we obtain the same result: strongly finitary monads bijectively correspond to varieties of complete algebras of $\CUMet$.
The proof is analogous, we indicate here what small changes are needed.

The category $\CUMet$ is closed symmetric monoidal: if $A$ and $B$ are complete spaces, then so are the spaces $A \tensor B$ and $[A,B]$ in Notation~\ref{N:met}.
In this section by 'category' and 'functor' we mean an enriched category (and functor) over $\CUMet$.
This category is complete and cocomplete (\cite{adamek+rosicky:approximate-injectivity}, Example 4.5).

\begin{remark}
	\phantom{phantom}
	\begin{enumerate}[label=(\alph*)]
		
		\item Recall the Cauchy completion $\cauchy{X}$ of an ultrametric space $X$: it is the (essentially unique) complete ultrametric space containing $X$ as a dense subspace.
		This embedding $X \hookrightarrow \cauchy{X}$ is a reflection of $X$ in $\CUMet$: for every nonexpanding map $f : X \to Y$ with $Y$ complete there is a unique nonexpanding extension $\cauchy{f} : \cauchy{X} \to Y$.
		Indeed, since $f$ is continuous, a unique continuous extension exists.
		To prove that $\cauchy{f}$ is indeed nonexpanding, consider $a, b \in \cauchy{X}$ with $d(a,b) = r$.
		Find sequences $(a_n)$, $(b_n)$ in $X$ converging to $a$ and $b$, respectively.
		Since $\cauchy{f}(a) = \lim_{n \to \infty} f(a_n)$ and $\cauchy{f}(b) = \lim_{n \to \infty} f(b_n)$, we get
		\begin{align*}
			d(\cauchy{f}(a),\cauchy{f}(b)) & = \lim_{n \to \infty}d(f(a_n),f(b_n)) \\
			& \leq \lim_{n \to \infty}d(a_n,b_n) \\
			& = d(a,b).
		\end{align*}
		
		\item Cauchy completion preserves finite products: the space $\cauchy{X} \times \cauchy{Y}$ is complete, and $X \times Y$ is clearly dense in it.
		Thus $\cauchy{X \times Y} \cong \cauchy{X} \times \cauchy{Y}$.
		
		\item For every natural number $n$ the endofunctor $(\blank)^n$ of $\CUMet$ preserves directed colimits.
		Indeed, directed colimits in $\CUMet$ are the Cauchy completions of the corresponding colimits in $\Met$.
		Thus our statement follows from Example~\ref{E:dirn} and (b) above.
		
		\item Directed colimits in $\CUMet$ have a characterization analogous to Proposition~\ref{P:pres}.
		However, Condition~1 must be weakened, as the following example demonstrates.
	\end{enumerate}
\end{remark}

\begin{example}
	In $\CUMet$ the subspace of the real line on the set $A = \{ 0 \} \cup \{\frac{1}{2}, \frac{1}{4}, \frac{1}{8}, \dots \}$ is a colimit of the $\omega$-chain of subspaces $A_n = \{ 2^{-k} \mid k=1,\dots,n \}$.
	However, $0$ does not lie in the image of any of the colimit maps.
\end{example}

\begin{proposition}
	\label{P:dircol}
	Let $(D_i)_{i \in I}$ be a directed diagram in $\CUMet$.
	A cocone $c_i : D_i \to C$ is a colimit iff
	\begin{enumerate}
		\item it is collectively dense: $C$ is the closure of $\bigcup_{i \in I} c_i[D_i]$, and
		\item for every $i \in I$, given $y,y' \in D_i$ we have
		\[
		d(c_i(y),c_i(y')) = \inf_{j \geq i} d(f_j(y),f_j(y'))
		\]
		where $f_j : D_i \to D_j$ is the connecting map.
	\end{enumerate}
\end{proposition}

\begin{proof}
	Denote by $c_i' : D_i \to C'$ a colimit of our diagram in $\Met$.
	Then a colimit in $\CUMet$ is given by the cocone $m \comp c_i': D_i \to C'$ where $m: C' \to C$ is the Cauchy completion of $C'$.
	\begin{enumerate}[label=(\alph*)]
		
		\item Sufficiency: If (1) and (2) hold, put $C' = \bigcup_{i \in I} c_i[D_i]$ and denote by $c_i' : D_i \to C'$ the codomain restriction of $c_i$ for each $i \in I$.
		Consider $C'$ as the metric subspace of $C$.
		The we obtain a cocone $(c_i')$ of our diagram in $\Met$.
		Since by our Condition (1) the subspace $C'$ is dense in $C$, we conclude that the inclusion map $m: C' \hookrightarrow C$ is a Cauchy completion of $C'$.
		Therefore, the cocone $c_i = m \comp c_i : D_i \to C$ is a colimit in $\CUMet$, as claimed.
		
		\item Necessity: If $c_i : D_i \to C$ is a colimit cocone in $\CUMet$ and $c_i': D_i \to C'$ is the cocone of codomain restrictions to $C' = \bigcup_{i \in I} c_i[D_i]$, then we can assume without loss of generality that the inclusion map $m : C' \to C$ is a Cauchy completion of $C'$ with $c_i = m \comp c'$ for all $ \in I$.
		Since $c_i'$ are collectively epic by Proposition~\ref{P:pres}, we conclude that $c_i$ are collectively dense.
		Since $c_i'$ fulfil Condition (2) of~\ref{P:pres} and $m$ is an isometric embedding, we conclude that $c_i$ fulfil that condition, too.
	\end{enumerate}
\end{proof}

\begin{definition}
	By a \emph{foliation} of a complete ultrametric space $M$ in $\CUMet$ is meant the weighted diagram $D_M': \B \to \CUMet$ with weight $B' : \B^\op \to \CUMet$ where $D_M'$ and $B'$ are the codomain restrictions of the corresponding functors in Definition~\ref{D:prec}.
\end{definition}

\begin{proposition}
	Every complete ultrametric space is the colimit of its foliation in $\CUMet$.
\end{proposition}

The proof is the same as that of Proposition~\ref{P:prec}.

\begin{notation}
	Every set (= discrete space) is complete.
	We thus get a full embedding
	\[
	K : \Set_\fp \hookrightarrow \CUMet
	\]
	of the category of finite sets and mappings into $\CUMet$.
\end{notation}

\begin{definition}
	An endofunctor $T$ of $\CUMet$ is \emph{strongly finitary} if it is the left Kan extension of its restriction $T \comp K$ to finite discrete spaces:
	\[
	T = \Lan{K}{(T \comp K)}.
	\]
\end{definition}

\begin{theorem}
	An endofunctor of $\CUMet$ is strongly finitary iff it is finitary and preserves colimits of foliations.
\end{theorem}

The proof is analogous to that of Theorem~\ref{T:sa}.
Indeed, every complete metric space is a directed colimit of its finite subspaces, and those are obtained as colimits of foliations (using finite discrete spaces).

\begin{notation}
	\phantom{phantom}
	\begin{enumerate}
		
		\item Analogously to Notation~\ref{N:fsf} the category $\Mon(\CUMet)$ of monads on $\CUMet$ is enriched via the supremum metric (inherited from $[\CUMet,\CUMet]$).
		Indeed, it is easy to verify that, given monads $\Tmon$ and $\Smon$, the subspace of $\CUMet(T,S)$ on all monad morphisms is closed, thus, complete.
		
		\item We denote by $\Mon_\fin(\CUMet)$ and $\Mon_\sf(\CUMet)$ the categories of finitary and strongly finitary monads on $\CUMet$, respectively (cf.\ Notation~\ref{N:fsf}).
		
	\end{enumerate}
\end{notation}

\begin{proposition}
	The categories $\Mon_\fin(\CUMet)$ and $\Mon_\sf(\CUMet)$ are closed under weighted limits in the category of $\aleph_1$-accessible monads on $\CUMet$.
\end{proposition}

The proof is the same as for Proposition~\ref{P:cls}.

\begin{remark}
	In $\Met$ we have used the factorization system (surjective, isometric embedding) (Remark~\ref{R:prec}).
	In $\CMet$ we have the factorization system (epi, strong mono): epimorphisms are the morphisms whose images are dense, and strong monomorphisms are those representing closed subspaces.
\end{remark}

\begin{proposition}
	\label{P:facto}
	Let $\Tmon$ be a finitary monad on $\CUMet$ preserving epimorphisms.
	In the category $\Mon_\fin(\CUMet)$ every morphism $h: \Tmon \to \Smon$ has a factorization $h = m \comp e$ where $e$ has epic components and $m$ has components representing closed subspaces.
\end{proposition}

\begin{proof}
	Since $T$ preserves $\EE$-morphisms of the factorization system (epi, strong mono), the fact that $h$ has a factorization in the category of all monads as described is proved as Items (1)-(3) of Proposition~\ref{P:fact}.
	It remains to prove that the functor $T'$ is finitary.
	That is, given a directed colimit $c_i : D_i \to C$ ($i \in I$) in $\CUMet$, we are to verify that the cocone $(T'c_i)_{i \in I}$ satisfies (1) and (2) of Proposition~\ref{P:dircol}.
	\begin{description}
		
		\item[\textnormal{Ad (1)}] Our task is to find for every $x \in T'\ol{C}$ and every $\eps > 0$ an element $y \in D_i$ (for some $i \in I$) with $d(x,T'c_i(y)) < \eps$.
		Form a colimit $\ol{c}_i : D_i \to \ol{C}$ ($i \in I$) of the given diagram in $\Met$.
		Without loss of generality $C$ is the Cauchy completion of $\ol{C}$ and $c_i = m \comp \ol{c}_i$ for the inclusion map $m : \ol{C} \to C$.
		Since $e_C: TC \to T'C$ is dense, there exists $x_0 \in T\ol{C}$ with
		\[
		d(x, e_{\ol{C}}(x_0)) < \frac{\eps}{2}.
		\]
		The functor $T$ is finitary, hence by Proposition~\ref{P:dircol} the cocone $(T\ol{c}_i)$ is collectively dense.
		For $x_0$ there thus exists $i \in I$ and $y \in TD_i$ with
		\[
		d(x_0, T\ol{c}_i(y)) < \frac{\eps}{2}.
		\]
		Apply the nonexpanding map $e_{\ol{C}}$ and use $e_{\ol{C}} \comp T\ol{c}_i = T' \ol{c}_i \comp e_{D_i}$:
		\[
		d(e_{\ol{C}}(x_0), T'\ol{c}_i \comp e_{D_i}(y_0)) \leq \eps.
		\]
		The desired element is $y = e_{D_i}(y_0)$: since $T'c_i(y) = T\ol{c}_i \comp e_{D_i}(y_0)$ we get
		\[
		d(x,T'c_i(y)) \leq d(x,e_{\ol{C}}(x_0)) + d(e_{\ol{C}}(x_0), Tc_i(y)) < \frac{\eps}{2} + \frac{\eps}{2}.
		\]
		
		\item[\textnormal{Ad (2)}] This is the same proof as in Proposition~\ref{P:fact}, using $m: T' \to S$.
	\end{description}
\end{proof}

\begin{notation}
	Let $\Sigma$ be a finitary signature.
	We denote by $\Sigma\text{-}\CUMet$ the category of complete ultra-quantitative algebras.
	This is the full subcategory of $\Sigma\text{-}\UMet$ on algebras whose underlying metric is complete.
\end{notation}

\begin{example}
	For every complete space $X$ the free $\Sigma$-algebra $T_\Sigma X$ of Example~\ref{E:term} is complete.
	This is then the free algebra in $\Sigma\text{-}\CUMet$ w.r.t.\ $\eta_X: X \to T_\Sigma X$.
	
	Indeed, we have seen in Corollary~\ref{C:term} that $T_\Sigma X$ is a coproduct of finite powers of $X$, thus the metric space $T_\Sigma X$ is complete.
	Its universal property in $\Sigma\text{-}\CUMet$ thus follows from Proposition~\ref{P:poly}.
\end{example}

\begin{corollary}
	The monad $\Tmon_\Sigma$ of free complete quantitative algebras is strongly finitary on $\CUMet$.
\end{corollary}

The argument is the same as in Corollary~\ref{C:term}.

\begin{definition}
	A \emph{variety} of complete quantitative algebras is a full subcategory of $\Sigma\text{-}\CMet$ specified by a set of quantitative equations.
\end{definition}

\begin{example}
	We describe the monad $\Tmon$ of free ultra-quantitative complete semilattices on $\CUMet$.
	It assigns to every complete ultrametric space $M$ the space $TM$ of all compact subsets with the Hausdorff metric (Example~\ref{E:mons}~(4)).
	
	This holds for separable complete spaces: see~\cite{mpp16}, Theorem~9.6.
	To extend this result to all complete spaces, first observe that the subset $Z$ of $TM$ of all finite sets is dense.
	Indeed, every compact set $K \subseteq M$ lies in the closure of $Z$: given $\eps > 0$, let $K_0 \subseteq K$ be a finite set such that $\eps$-balls with centers in $K_0$ cover $K$.
	Then $K_0 \in Z$ and the Hausdorff distance of $K_0$ and $K$ is at most $\eps$.
	
	Given a complete ultrametric space $M$, let $X_i$ ($i \in I$) be the collection of all countable subsets.
	Each closure $\ol{X}_i$ is a complete separable space, and $M = \bigcup_{i \in I} \ol{X}_i$ is a directed colimit (see Remark~\ref{R:new}) preserved by $T$.
	Since $T\ol{X}_i$ is the space of all compact subsets of $\ol{X}_i$, and since finite subsets of $M$ form a dense set, we conclude that $TM$ is the space of all compact subsets of $M$.
\end{example}

\begin{proposition}
	Let $\Vvar$ be a variety of complete ultra-quantitative algebras.
	\begin{enumerate}
		
		\item $\Vvar$ has free algebras: the forgetful functor $U_\Vvar : \Vvar \to \CUMet$ has a left adjoint $F_\Vvar: \CUMet \to \Vvar$.
		
		\item For the corresponding monad $\Tmon_\Vvar$ with $T_\Vvar = U_\Vvar \comp F_\Vvar$ the categories $\V$ and $\CUMet^{\Tmon_\Vvar}$ are concretely isomorphic (via the comparison functor).
	\end{enumerate}
\end{proposition}

The proof is analogous to that of Corollary~\ref{C:left} and Proposition~\ref{P:acc}.

\begin{construction}
	For every variety $\Vvar$ of complete ultra-quantitative algebras we construct a weighted diagram in $\Mon_\sf(\CUMet)$ completely analogous to Construction~\ref{C:diagram}.
	We just understand the objects $D_i = \Tmon_{[n(i)]}$ and $Da = \Tmon_\Sigma$ as monads on $\CUMet$.
\end{construction}

\begin{theorem}
	Let $\Vvar$ be a variety of complete ultra-quantitative algebras.
	Then the free-algebra monad $\Tmon_\Vvar$ is the colimit of the above weighted diagram in $\Mon_\sf(\CUMet)$.
\end{theorem}

The proof is analogous to that of Theorem~\ref{T:var-mon}, using Proposition~\ref{P:facto} and the fact that discrete spaces are complete.

\begin{corollary}
	For every variety $\Vvar$ of complete ultra-quantitative algebras the monad $\Tmon_\Vvar$ is strongly finitary.
\end{corollary}

Recall the variety $\Vvar_\Tmon$ assigned to every strongly finitary monad $\Tmon$ on $\Met$ in Construction~\ref{C:eq}.
We can define a variety $\Vvar_\Tmon$ of complete ultra-quantitative algebras by the same signature and the same equations.
Using the same proof as that of Theorem~\ref{T:sfv} we obtain the following

\begin{theorem}
	Every strongly finitary monad on $\CUMet$ is the free-algebra monad of the variety $\Vvar_\Tmon$.
\end{theorem}

\begin{corollary}
	The category of finitary varieties of complete ultra-quantitative algebras (and concrete functors) is dually equivalent to the category $\Mon_\sf(\CUMet)$ of strongly finitary monads.
\end{corollary}

\section{Infinitary algebras}
\label{S:inf-algs}

The above bijective correspondence holds, more generally, for $\lambda$-ary ultra-quantitative algebras, where $\lambda$ is an arbitrary infinite regular cardinal.
The proof is completely analogous, we indicate the (small) modifications needed in the present section.

\begin{assumption}
\label{A:as}
Throughout the rest of the paper $\lambda$ denotes a regular infinite cardinal.
And a standard set $V = \{ x_i \mid i < \lambda \}$ of variables is chosen.
\end{assumption}

Let $\Sigma$ be a $\lambda$-ary signature: a set of operation symbols of arities which are cardinals $n < \lambda$.
A quantitative $\Sigma$-algebra is a metric space $A$ together with nonexpanding maps $\sigma_A: A^n \to A$ for every $n$-ary symbol $\sigma \in \Sigma$.
We again obtain the category $\Sigma\text{-}\Met$ of quantitative algebras and nonexpanding homomorphisms and $\Sigma\text{-}\UMet$ as its subcategory of ultra-quantitative algebras.

\begin{example}
The free algebra on a space $X$ is the algebra $T_\Sigma X$ of terms (defined analogously to Example~\ref{E:term}).
If $X$ is discrete, then so is $T_\Sigma X$.
\end{example}

Quantitative equations are defined precisely as in Definition~\ref{D:qe}.
Also the definition of a variety of $\Sigma$-algebras is unchanged.

\begin{example}
\label{E:sigma-semilat}
Recall that a \emph{$\sigma$-semilattice} is a poset with countable joins.
We can express it by adding to the operations $+$ and $0$ of Example~\ref{E:quant}(3) an operation $\bigsqcup$ of arity $\omega$.
Besides the semilattice equations for $+$ and $0$ one needs the following associativity equation:
\[
y + \bigsqcup_{n < \omega} x_n = \bigsqcup_{n < \omega} (y + x_n),
\]
idempotence of $\bigsqcup$:
\[
\bigsqcup_{n < \omega} x = x,
\]
and for every $k < \omega$ the equation
\[
x_k + \bigsqcup_{n < \omega} x_n = \bigsqcup_{n < \omega} x_n.
\]
When we then define $x \leq y$ by $x + y = y$, it is not difficult to verify that $\bigsqcup_{n < \omega} x_n$ is the join of $\{ x_n \mid n < \omega \}$.
A quantitative $\sigma$-semilattice is then a $\sigma$-semilattice on a metric space with $+$ and $\bigcup$ nonexpanding.

The \emph{monad of quantitative $\sigma$-semilattices} is given by the metric spaces $TM = \Pow_{\omega_1} M$ of countable subsets of $M$ with the Hausdorff metric.
The proof is analogous to the proof for semilattices in~\cite{mpp16}.
\end{example}

\begin{theorem}
An enriched endofunctor $T$ of $\Met$ is strongly $\lambda$-accessible (Definition~\ref{D:sa}) iff it
\begin{enumerate}
	\item is $\lambda$-accessible (preserves $\lambda$-directed colimits) and
	\item preserves colimits of foliations.
\end{enumerate}
\end{theorem}
The proof is essentially the same as that of Theorem~\ref{T:sa}.

\begin{lemma}
\label{L:new}
A $\Sigma$-algebra expressed by $\alpha : T_\Sigma A \to A$ satisfies $l =_\eps r$ iff $d(\wh{\alpha} \comp \ol{l}, \wh{\alpha} \comp \ol{r}) \leq \eps$.
\end{lemma}

\begin{proof}
This is completely analogous to the proof of Lemma~\ref{L:free}.
In the proof of the inequality~\eqref{eq:proof} we proceed by transfinite induction on the height $h(t)$.
This is an ordinal number defined by $h(x) = 0$ for variables $x$, and $h(\sigma(t_i)_{i<n}) = 1 + \bigvee_{i<n} h(t_i)$.
\end{proof}

\begin{theorem}
Every variety of $\lambda$-ary ultra-quantitative algebras is concretely isomorphic to $\UMet^\Tmon$ for a strongly $\lambda$-accessible monad $\Tmon$.
\end{theorem}

The proof is completely analogous to that of Corollary~\ref{C:cor-ultra}: we construct a weighted diagram in $\Mon_\lambda(\UMet)$ as in Construction~\ref{C:diagram}.
Its colimit $\Tmon$ exists in $\Mon_\lambda(\UMet)$ by Proposition~\ref{P:lambda-acc}.
And the proof that $\Tmon$ is the free-algebra monad for $\Vvar$ follows using Lemma~\ref{L:new}, the same steps as the proof of Theorem~\ref{T:var-mon}.

\begin{construction}
Given a strongly $\lambda$-accessible monad $\Tmon$ on $\UMet$ we define a $\lambda$-ary signature $\Sigma$ by $\Sigma_n = |Tn|$ for all cardinals $n < \lambda$.
The variety associated with $\Tmon$ is given by equations $(1)$-$(3)$ of Construction~\ref{C:eq} (where $n$ and $m$ range over cardinals smaller than $\lambda$).
\end{construction}

\begin{theorem}
Every strongly $\lambda$-accessible monad on $\UMet$ is the free-algebra monad of its associated variety.
\end{theorem}

The proof is analogous to that of Theorem~\ref{T:sfv}.

\begin{corollary}
\label{C:dualeq}
The following categories are dually equivalent:
\begin{enumerate}
	\item varieties of $\lambda$-ary ultra-quantitative algebras (and concrete functors), and
	\item strongly $\lambda$-accessible monads on $\UMet$ (and monad morphisms).
\end{enumerate}
\end{corollary}

\begin{remark}
Assuming that strongly $\lambda$-accessible monads on $\Met$ are closed under composition, the analogous corollary holds for them: they correspond bijectively to varieties of $\lambda$-ary quantitative algebras.
\end{remark}

\section{$\lambda$-Basic Monads and $\lambda$-Basic Varieties}
\label{S:basic}

In this section we consider more general equations than those in Section~\ref{S:inf-algs}: $\lambda$-basic quantitative equations introduced by Mardare et al.~\cite{mpp17}.
And we show that the corresponding varieties of quantitative algebras bijectively correspond to monads on $\Met$ which are \emph{$\lambda$-basic}, i.e.\ enriched, $\lambda$-accessible and preserving surjective morphisms.
For $\lambda = \aleph_1$ this has been proved in~\cite{adamek:varieties-quantitative-algebras}, the proof for $\lambda > \aleph_1$ is completely analogous, we thus only shortly indicate it here.
The case $\lambda = \aleph_0$ is an open problem that we discuss below.

The idea of $\lambda$-basic equations is simple and natural: in Section~\ref{S:inf-algs} we have considered equations $l =_\eps r$ where $l$ and $r$ are terms in $T_\Sigma V$, the discrete free algebra on the set of variables.
Why restrict ourselves to discrete free algebras?
We can also consider $l =_\eps r$ for pairs $l$,$r$ of elemets of $T_\Sigma M$ for an \emph{arbitrary} metric space $M$.
Now restricting to $M$ of power less than $\lambda$ (the arity of the signature $\Sigma$) does not make any difference: every metric space $M_0$ is a $\lambda$-directed colimit of its subspaces $M \subseteq M_0$ of power less than $\lambda$, and $T_\Sigma$ preserves directed colimits.
Thus every pair $l,r \in T_\Sigma M_0$ lies in $T_\sigma M$ with $\card M < \lambda$.
This leads us to the following definition, formulated in~\cite{mpp17} for $\lambda = \aleph_0$ or $\aleph_1$.

\begin{assumption}
Recall Assumption~\ref{A:as}.
Further recall Notation~\ref{N:ext} $f^\sharp: T_\Sigma M \to A$ for the unique homomorphism extending $f: M \to A$.
For the rest of our paper $\lambda$ denotes an uncountable regular cardinal.
And $V$ is a chosen set of variables of cardinality $\lambda$.
\end{assumption}

\begin{definition}[\cite{mpp17}]
\label{D:basic}
Let $\Sigma$ be a signature of arity $\lambda$.
\begin{enumerate}
	\item A \emph{$\lambda$-basic quantitative equation} is an expression
	\[
	M \vdash l =_\eps r
	\]
	where $M$ is a metric space (of variables) with $\card M < \lambda$, and $l,r$ are terms in $T_\Sigma M$.
	For $\lambda = \aleph_0$ we speak about $\omega$-basic equations.
	
	\item A quantitative algebra $A$ \emph{satisfies} this equation if every nonexpanding interpretation of variables $f: M \to A$ fulfils
	\[
	d(f^\sharp(l),f^\sharp(r)) \leq \eps.
	\]
\end{enumerate}
\end{definition}

\begin{example}
\label{E:basic}
\phantom{phantom}
\begin{enumerate}
	\item \emph{Quasi-commutative monoids} are quantitative monoids in which the commutative law holds for all pairs of distance at most 1.
	They can be presented by the usual monoid equations plus the following $\omega$-basic equation
	\[
	x =_1 y \vdash x * y = y * x.
	\]
	More precisely: the left-hand side is the space $M = \{ x, y \}$ with $d(x,y) = 1$.
	
	\item \emph{Almost quasi-commutative monoids} (compare Example~\ref{E:quant}~(2)): here we only request that the distance of $ab$ and $ba$ is at most $\eps$ for pairs of distance at most 1:
	\[
	x =_1 y \vdash x * y =_\eps y * x.
	\]
	
	\item \emph{Quasi-discrete spaces} (see~\cite{rosicky:metric-monads}) are metric spaces with all non-zero distances greater than 1.
	Here we put $\Sigma = \emptyset$ and consider the $\omega$-basic equation
	\[
	x =_1 y \vdash x = y.
	\]
\end{enumerate}
\end{example}

\begin{remark}
\phantom{phantom}
\label{R:bas}
\begin{enumerate}
	\item The above $\omega$-basic equations work with finite spaces $M$.
	Suppose we list all pairs of distinct elements as $(x_0,y_0)$, \dots, $(x_{n-1},y_{n-1})$ and denote by $\delta_i$ the distance $d(x_i,y_i)$ in $M$.
	Then, as we have seen in the above examples, we can rewrite $M \vdash l =_\eps r$ in the form
	\[
	(x_0 =_{\delta_0} y_0), \dots, (x_{n-1} =_{\delta{n-1}} y_{n-1}) \vdash l =_\eps r.
	\]
	This is the syntax presented in~\cite{mpp17}.
	In fact, in that paper the syntax is a bit more general: the left-hand side of $\vdash$ need not be connected to a concrete space $M$.
	However, this does not influence the expressivity of the $\omega$-basic equations, see Remark~\ref{R:exp} below.
	
	\item Analogously, $\omega_1$-basic equations, working with countable spaces $M$, can be rewritten in the form
	\begin{equation*}
	\label{eq:MPP} \tag{{+}}
	(x_0 =_{\delta_0} y_0), (x_1 =_{\delta_1} y_1), (x_2 =_{\delta_2} y_2), \dots \vdash l =_\eps r.
	\end{equation*}
\end{enumerate}
\end{remark}

\begin{definition}
Given a $\lambda$-ary signature $\Sigma$, by a \emph{$\lambda$-basic variety} is meant a full subcategory of $\Sigma\text{-}\Met$ which can be presented by a set of $\lambda$-basic quantitative equations.
\end{definition}

For $\lambda = \omega$ or $\omega_1$ this definition stems from~\cite{mpp17}.
There it is also proved that every $\lambda$-basic variety has free algebras (their proof works for all $\lambda$).
We can thus introduce the corresponding notation:

\begin{notation}
For every $\lambda$-basic variety $\Vvar$ we denote by $\Tmon_\Vvar$ the free-algebra monad on $\Met$.
\end{notation}

The following Birkhoff Variety Theorem was also formulated in~\cite{mpp17}.
The proof there is not correct. See B19-B20 in the paper of Milius and Urbat~\cite{mu} for a correct proof.

\begin{definition}[\cite{mpp17}]
An epimorphism $e : X \to Y$ in $\Met$ is called \emph{$\lambda$-reflexive} provided that for every subspace $Y_0 \subseteq Y$ with $\card Y_0 < \lambda$ there exists a subspace $X_0 \subseteq X$ such that $e$ restricts to an isomorphism $e_0: X_0 \xrightarrow{\sim} Y_0$ in $\Met$.
\end{definition}

\begin{theorem}[Birkhoff Variety Theorem]
Let $\Sigma$ be a signature of arity $\lambda$.
A full subcategory of $\Sigma\text{-}\Met$ is a $\lambda$-variety iff it is closed under products, subalgebras and $\lambda$-reflexive homomorphic images.
\end{theorem}

\begin{proposition}
Every $\lambda$-variety $\Vvar$ of quantitative algebras is concretely isomorphic to $\Met^{\Tmon_\Vvar}$.
\end{proposition}

This follows from~\cite{Ro-new}, Section~4.

\begin{proposition}
\label{P:sur}
For every $\lambda$-basic variety the monad $\Tmon_\Vvar$ is enriched and preserves surjective morphisms.
\end{proposition}

\begin{proof}
\phantom{phantom}
\begin{enumerate}

	\item For $\Vvar = \Sigma\text{-}\Alg$ the corresponding monad $\Tmon_\Sigma$ has both properties.
	Indeed, $T_\Sigma M$ is the algebra of terms precisely as in the finitary case (Example~\ref{E:term}).
	It follows that $T_\Sigma$ is a coproduct of functors $(\blank)^n$ for $n < \lambda$, one for every similarity class of terms in $T_\Sigma V$ (compare Corollary~\ref{C:sf}).
	Due to the Birkhoff Variety Theorem for every space $M$ there is a reflection $\eps_M: T_\Sigma M \to T_\Vvar M$ of the free $\Sigma$-algebra in $\Vvar$, and $\eps_M$ is surjective.
	
	\item We obtain a monad morphism $\eps: \Tmon_\Sigma \to \Tmon_\Vvar$ with surjective components.
	This implies that $\Tmon_\Vvar$ inherits the above two properties from $\Tmon_\Sigma$.

\end{enumerate}
\end{proof}

\begin{definition}
An (enriched) monad on $\Met$ is called \emph{$\lambda$-basic} if it is $\lambda$-accessible and preserves surjective morphisms.
\end{definition}

\begin{theorem}
\label{T:basic}
Let $\lambda$ be an uncountable regular cardinal.
Then the free-algebra monad of every $\lambda$-variety is $\lambda$-basic.
\end{theorem}

\begin{proof}
\phantom{phantom}
\begin{enumerate}

	\item We only need to verify that $\Tmon_\Vvar$ is $\lambda$-accessible.
	This is easy to see in case $\Vvar = \Sigma\text{-}\Alg$ since every term in $T_\Sigma M$ contains less than $\lambda$ variables.
	Thus the forgetful functor $U : \Sigma\text{-}\Alg \to \Met$ preserves $\lambda$-filtered colimits, using Corollary~\ref{C:count}.
	Since its left adjoint $F$ preserves colimits, $T_\Sigma = UF$ is $\lambda$-accessible.
	
	\item We verify that every $\lambda$-basic variety $\Vvar$ is closed under $\lambda$-filtered colimits in $\Sigma\text{-}\Met$.
	To prove this, consider a single $\lambda$-basic equation $M \vdash l =_\eps r$.
	We verify that given a $\lambda$-directed diagram of quantitative algebras $A_i$ ($i \in I$) satisfying that equation, it follows that $A = \colim_{i \in I} A_i$ also satisfies it.
	
	For every interpretation $f: M \to A$ we have, since $M$ is $\lambda$-presentable in the enriched sense (Proposition~\ref{P:presen}), a factorization $g$ through the colimit map $c_i : A_i \to A$ for some $i \in I$ in $\Met$:
	\[
	\begin{tikzcd}
	& A_i \arrow[d, "c_i"] \\
	M \arrow[ur, dashed, "g"] \arrow[r, swap, "f"] & A
	\end{tikzcd}
	\]
	Then $f^\sharp = c_i \comp g^\sharp$ because $c_i$ is a morphism in $\Sigma\text{-}\Met$.
	Since $A_i$ satisfies $M \vdash l =_\eps r$, we have $d(g^\sharp(l),g^\sharp(r)) \leq \eps$.
	This implies $d(f^\sharp(l),f^\sharp(r)) \leq \eps$ because $c_i$ is nonexpanding.
	
	\item Thus, the forgetful functor $U_\Vvar : \Vvar \to \Met$ preserves $\lambda$-filtered colimits, too.
	Since $T_\Vvar = U_\Vvar \comp F_\Vvar$ where $F_\Vvar \adj U_\Vvar$ preserves all colimits, this finishes the proof.
	
\end{enumerate}
\end{proof}

\begin{example}[\cite{adamek:varieties-quantitative-algebras}]
\label{E:contra}
Unfortunately, for $\omega$-varieties $\Vvar$ the monads $\Tmon_\Vvar$ are not finitary in general.
A simple example of this are the quasi-discrete spaces (Example~\ref{E:basic} (3)).
We demonstrate that $\Tmon_\Vvar$ does not preserve colimits of $\omega$-chains of subspaces.
Denote by $A_n$ the subspace of the real line consisting of $-1$ and $2^{-k}$ for $k=0,\dots,n$.
Form the $\omega$-chain of inclusions $A_n \hookrightarrow A_{n+1}$.
Its colimit is the subspace $A = \{ -1 \} \cup \{ 2^{-k} \mid k \in \Nat \}$.
The functor $T_\Vvar$ assigns to $A_0 = \{ -1, 1\}$ the 2-element space with distance $2$, to $A_1 = \{ -1, \frac{1}{2}, 1\}$ the $2$-element space with distance $\frac{3}{2}$ etc.
We see that in $T_\Vvar A_n$ the two elements have distance $1 + 2^{-n}$.
Consequently, $\colim_{n < \omega} T_\Vvar A_n$ is the 2-element space with distance $1$.
This space does not lie in $\Vvar$, thus it is not $T_\Vvar(\colim A_n)$.
In fact, the last space has one element.
\end{example}

\begin{notation}
\label{N:lambda2}
In the following construction $\Met_\lambda$ is a set of metric spaces representing all spaces $M$ with $\card |M| < \lambda$ up to isomorphism.
We assume for simplicity that $\Met_\lambda$ contains the discrete spaces $V_k = \{ x_n \mid n < k \}$ for all cardinals $k < \lambda$.
\end{notation}

\begin{construction}[\cite{adamek:varieties-quantitative-algebras}]
Let $\lambda$ be a regular cardinal and $\Tmon$ a $\lambda$-basic monad.
We define a $\lambda$-basic variety of algebras of the following signature $\Sigma$:
\[
\Sigma_n = |TV_n| \text{ for all cardinals } n < \lambda.
\]
It is presented by the following $\lambda$-basic equations where $n$ and $m$ denote arbitrary cardinals smaller than $\lambda$ (analogous to Construction~\ref{C:eq}):
\begin{enumerate}
	\item $M \vdash l =_\eps r$ for all $M \in \Met_\lambda$ and $l, r \in |TM|$ of distance $\eps$ in $TM$,
	\item $k^*(\sigma) = \sigma(k(x_i))$ for all $\sigma \in TV_n$ and all maps $k: V_n \to |TV_m|$,
	\item $\eta_{V_n}(x_i) = x_i$ for all $n < \lambda$ and $i < n$.
\end{enumerate}
\end{construction}

\begin{theorem}[\cite{adamek:varieties-quantitative-algebras}]
\label{T:basic2}
Every $\lambda$-basic monad on $\Met$ is the free-algebra monad of the $\lambda$-basic variety of the above construction.
\end{theorem}

For $\lambda = \aleph_0$ or $\aleph_1$ this has been proved in~\cite{adamek:varieties-quantitative-algebras}, Theorem~3.11.
The general case is completely analogous.

\begin{theorem}
\label{T:main2}
For every uncountable regular cardinal $\lambda$ the following categories are dually equivalent:
\begin{enumerate}
	\item $\lambda$-basic varieties of quantitative algebras (and concrete functors), and
	\item $\lambda$-basic monads (and monad morphisms).
\end{enumerate}
\end{theorem}

This follows from Theorems~\ref{T:basic} and~\ref{T:basic2} precisely as we have seen it in Theorem~\ref{T:main}.

\begin{remark}
The development above works completely analogously for the full subcategory $\UMet$.
Thus $\lambda$-basic varieties of ultra-quantitative algebras bijectively correspond to $\lambda$-basic monads on $\UMet$.
\end{remark}

\begin{openproblem}
Which monads correspond to $\omega$-basic varieties?
\end{openproblem}

As we have seen in Example~\ref{E:contra} these monads are not finitary, they can even fail to preserve directed colimits of isometric embeddings.
A partial solution has been presented in~\cite{adamek:varieties-quantitative-algebras}:
$\omega$-basic varieties are, up to isomorphism, precisely the categories $\UMet^\Tmon$, where $\Tmon$ is an enriched monad preserving surjective morphisms and directed colimits of split subobjects.

\begin{remark}
\label{R:exp}
We now show that, for signatures of arity $\aleph_1$, as considered in~\cite{mpp17}, our $\omega_1$-basic equations are precisely as expressive as the formulas~\eqref{eq:MPP} of Remark~\ref{R:bas}~(2).
An algebra $A$ \emph{satisfies} the latter formula provided that every interpretation $f : V \to A$ of variables for which $d(f(x_i),f(y_i)) \leq \delta_i$ holds ($i \in \Nat$) fulfils $d(f^\sharp(l),f^\sharp(r)) \leq \eps$.

Below we work with (extended) \emph{pseudometrics}: we drop the condition that $d(x,y) = 0$ implies $x = y$.
\end{remark}

\begin{construction}
Let $\Sigma$ be an $\aleph_1$-ary signature.
For every formula~\eqref{eq:MPP} of Remark~\ref{R:bas}~(2) denote by $\wh{d}$ the smallest pseudometric on the set $V = \{ x_i \mid i \in \Nat \}$ of variables satisfying $\wh{d}(x_i,x_j) \leq \delta_i$ for $i,j \in \Nat$.
Let
\[
e: (V,\wh{d}) \to M
\]
be the metric reflection: the quotient of $(V,\wh{d})$ modulo the equivalence meging $z,z' \in V$ iff $\wh{d}(z,z') = 0$.
Put $\ol{l} = T_\Sigma e(l)$ and $\ol{r} = T_\Sigma e(r)$ to obtain the following $\omega_1$-basic equation:
\[
M \vdash \ol{l} =_\eps \ol{r}.
\]
\end{construction}

\begin{lemma}
An algebra $A$ satisfies $M \vdash \ol{l} =_\eps \ol{r}$ iff it satisfies the formula~\eqref{eq:MPP}.
\end{lemma}

\begin{proof}
Every interpretation $f : V \to A$ with $d(f(x_i),f(y_i)) \leq \delta_i$ ($i \in \Nat$) is a nonexpanding map $f : (V, \wh{d}) \to (A,d)$.
Indeed, define a pseudometric $d_0$ on $V$ by $d_0(z,z') = d(f(z),f(z'))$.
Then $\wh{d} \leq d_0$ by the minimality of $\wh{d}$.
Thus
\[
\wh{d}(z,z') \leq d_0(z,z') = d(f(z),f(z')).
\]
Therefore we have a unique nonexpanding map $g : M \to (A,d)$ with $f = g \comp e$:
\[
\begin{tikzcd}[column sep=small]
T_\Sigma V \arrow[rrrr, "T_\Sigma e"] \arrow[rrdd, bend right, swap, "f^\sharp"] & & & & T_\Sigma M \arrow[ddll, bend left, "g^\sharp"] \\
& V \arrow[ul, swap, "\eta_V"] \arrow[rr, "e"] \arrow[rd, "f"] & &  M \arrow[ur, "\eta_M"] \arrow[ld, swap, "g"] & \\
& & A & &
\end{tikzcd}
\]
As $T_\Sigma e$ is a nonexpanding homomorphism, we conclude that $f^\sharp = g^\sharp \comp T_\Sigma e$.

If $A$ satisfies $M \vdash \ol{l} =_\eps \ol{r}$, then we get $d(g^\sharp(\ol{l}),g^\sharp(\ol{r})) \leq \eps$ which implies $d(f^\sharp(l),f^\sharp(r)) \leq \eps$.
Conversely, if $A$ satisfies the formula~\eqref{eq:MPP}, then for every nonexpanding map $g : M \to A$ we put $f = g \comp e$ and define $d_0$ as above.
Since both $e$ and $g$ are nonexpanding and $\wh{d}(x_i,y_i) \leq \delta_i$, we conclude $d(f(x_i),f(y_i)) \leq \delta_i$ for $i \in \Nat$.
Thus $d(f^\sharp(l),f^\sharp(r)) \leq \eps$.
From $f^\sharp = T_\Sigma e \comp g^\sharp$ we get $d(g^\sharp(\ol{l}),g^\sharp(\ol{r})) \leq \eps$.
\end{proof}

\section{Accessible Monads and Generalized Varieties}
\label{S:am}

We now turn to general accessible enriched monads on $\Met$: they bijectively correspond to generalized varieties of quantitative algebras.
Recall our standing assumption that an uncountable regular cardinal $\lambda$ is given.
Our signature $\Sigma$ here will be more general than in Sections~\ref{S:inf-algs} and~\ref{S:basic}: arities are spaces, not cardinals.
Our treatment in this section follows ideas of Kelly and Power~\cite{kelly+power:adjunctions} who showed how to present $\lambda$-accessible monads on locally $\lambda$-presentable categories.
(They treated the case $\lambda = \aleph_0$ only, but their results immediately generalize to all $\lambda$.)
The basic difference of our approach is that our generalized signatures $\Sigma$ are collections $(\Sigma_M)_{M \in \Met_\lambda}$ of \emph{sets} $\Sigma_M$ of operation symbols of arity $M$, whereas in loc.\ cit.\ $\Sigma_M$ are arbitrary metric spaces.
This explains why plain equations were sufficient in loc.\ cit., whereas we are going to apply $\lambda$-basic equations analogous to Section~\ref{S:basic}.
Recall $\Met_\lambda$ from Notation~\ref{N:lambda2} and recall that $[M,A]$ is the space of all morphisms $f : M \to A$ with the supremum metric.

\begin{definition}
\phantom{phantom}
\begin{enumerate}
	\item A \emph{generalized signature} of arity $\lambda$ is a collection $\Sigma = (\Sigma_M)_{M \in \Met_\lambda}$ of sets $\Sigma_M$ of operation symbols of arity $M$.
	\item A \emph{quantitative algebra} is a metric space $A$ together with nonexpanding maps
	\[
	\sigma_A: [M,A] \to A \text{ for all } \sigma \in \Sigma_M.
	\]
	A \emph{homomorphism} to an algebra $B$ is a nonexpanding map $h : A \to B$ such that the squares below
	\[
	\begin{tikzcd}
	{[M,A]} \arrow[r, "\sigma_A"] \arrow[d, swap, "h \comp (\blank)"] & A \arrow[d, "h"] \\
	{[M,B]} \arrow[r, swap, "\sigma_B"] & B
	\end{tikzcd}
	\]
	commute for all $M \in \Met_\lambda$ and $\sigma \in \Sigma_M$.
	The category of $\Sigma$-algebras is again denoted by $\Sigma\text{-}\Met$.
\end{enumerate}
\end{definition}

\begin{example}
\label{E:bin}
Let $P = \{ x, y \}$ be the space in $\Met_\lambda$ with $d(x,y) = 1$.
Denote by $\Sigma$ the generalized signature with
\[
\Sigma_P = \{ \sigma \} \text{ and } \Sigma_\One = \{ s \}
\]
(where $\One$ is the singleton space), while all the other sets $\Sigma_M$ are empty.
An algebra is a metric space $A$ together with an element $s_A \in A$ and a partial binary operation
\[
\sigma_A : A \times A \rightharpoonup A
\]
with $\sigma(a_1,a_2)$ defined iff $d(a_1,a_2) \leq 1$.
\end{example}

\begin{remark}
\label{R:gn}
The forgetful functor $U : \Sigma\text{-}\Met \to \Met$ has a left adjoint analogous to Example~\ref{E:term}.
For every space $X$ we define the free algebra $T_\Sigma X$ on it as the algebra of \emph{terms}.
This is the smallest metric space such that
\begin{enumerate}
	\item $X$ is a subspace of $T_\Sigma X$, and
	\item given $\sigma$ of arity $M$ and an $|M|$-tuple of terms $t_i$ ($i \in M$) such that the map $M \to T_\Sigma X$ given by $i \mapsto t_i$ is nonexpanding, we obtain a composite term $\sigma(t_i)_{i \in M}$.
\end{enumerate}
Recall the similarity of terms from Example~\ref{E:term}.
The distance of non-similar terms is again $\infty$, and for similar terms $\sigma(t_i)_{i \in M}$ and $\sigma(t_i')_{i \in M}$ it is the supremum of $\{ d(t_i,t_i') \mid i \in |M| \}$.
The proof that $T_\Sigma X$ above is well defined and forms the free algebra on $X$ is analogous to Proposition~\ref{P:poly}: we just need to define the polynomial functor by
\[
H_\Sigma X = \coprod_{\sigma \in \Sigma} [M,X]
\]
for the arity $M$ of $\sigma$.
\end{remark}

\begin{example}
\label{E:bin2}
Let $\Sigma$ be the generalized signature of Example~\ref{E:bin}.
We first describe $T_\Sigma X$ for discrete spaces $X$.
We can identify elements $x \in X$ with singleton trees labelled by $x$, and composite terms $t = \sigma(t_l,t_r)$ with binary ordered trees having the maximum subtrees $t_l$ and $t_r$.
We see that $T_\Sigma X$ consists of all finite uniform binary trees (with the left-hand subtree similar to the right-hand one) where all the leaves are labelled by the same element $x$ of $X + \{ s \}$:
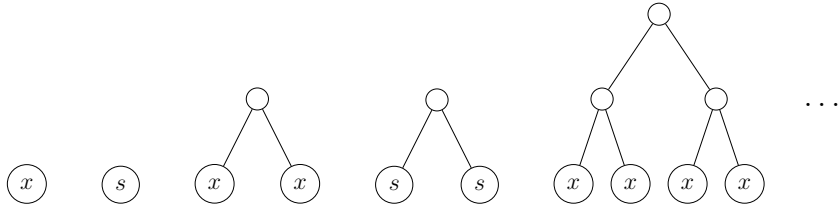
\[
\scalebox{0.75}{
\begin{tikzpicture}
\node[circle,draw] {$x$};
\end{tikzpicture}
\qquad
\begin{tikzpicture}
	\node[circle,draw] {$s$};
\end{tikzpicture}
\qquad
\begin{tikzpicture}
	\node[circle, draw] {}
	child {node[circle, draw] {$x$}}
	child {node[circle, draw] {$x$}};
\end{tikzpicture}
\qquad
\begin{tikzpicture}
	\node[circle, draw] {}
	child {node[circle, draw] {$s$}}
	child {node[circle, draw] {$s$}};
\end{tikzpicture}
\qquad
\begin{tikzpicture}
	[level 1/.style={sibling distance=20mm},
	level 2/.style={sibling distance=10mm}]
	\node[circle, draw] {}
	child {node[circle, draw] {}
		child {node[circle, draw] {$x$}}
		child {node[circle, draw] {$x$}}
	}
	child {node[circle, draw] {}
		child {node[circle, draw] {$x$}}
		child {node[circle, draw] {$x$}}
	};
\end{tikzpicture}
}
\quad
\raisebox{1.3cm}{\dots}
\]
And the space $T_\Sigma X$ is discrete.

For the space $X = \{ x, y \}$ with $d(x,y) = \frac{1}{2}$ the algebra $T_\Sigma X$ consists of all uniform binary trees where the leaves are either labelled by elements of $X$, or they all have the label $s$:
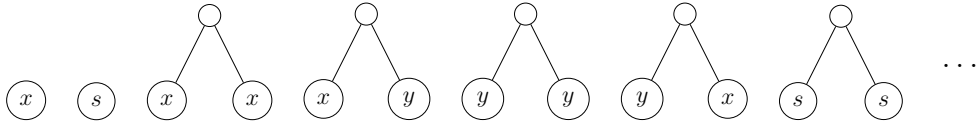
\[
\scalebox{0.75}{
	\begin{tikzpicture}
		\node[circle,draw] {$x$};
	\end{tikzpicture}
	\quad
	\begin{tikzpicture}
		\node[circle,draw] {$s$};
	\end{tikzpicture}
	\quad
	\begin{tikzpicture}
		\node[circle, draw] {}
		child {node[circle, draw] {$x$}}
		child {node[circle, draw] {$x$}};
	\end{tikzpicture}
	\quad
	\begin{tikzpicture}
		\node[circle, draw] {}
		child {node[circle, draw] {$x$}}
		child {node[circle, draw] {$y$}};
	\end{tikzpicture}
	\quad
	\begin{tikzpicture}
	\node[circle, draw] {}
	child {node[circle, draw] {$y$}}
	child {node[circle, draw] {$y$}};
	\end{tikzpicture}
	\quad
	\begin{tikzpicture}
	\node[circle, draw] {}
	child {node[circle, draw] {$y$}}
	child {node[circle, draw] {$x$}};
	\end{tikzpicture}
	\quad
	\begin{tikzpicture}
		\node[circle, draw] {}
		child {node[circle, draw] {$s$}}
		child {node[circle, draw] {$s$}};
	\end{tikzpicture}
}
\quad
\raisebox{0.7cm}{\dots}
\]
\end{example}

\begin{definition}
\label{D:bass}
Let $\Sigma$ be a generalized $\lambda$-ary signature.
A \emph{$\lambda$-basic quantitative equation} is, as in Definition~\ref{D:basic}, an expression $M \vdash l =_\eps r$ where $M \in \Met_\lambda$ and $l,r \in |T_\Sigma M|$.
An algebra $A$ \emph{satisfies it} if given a nonexpanding interpretation $f : M \to A$, we have
\[
d(f^\sharp(l),f^\sharp(r)) \leq \eps.
\]
A \emph{generalized variety} is a full subcategory of $\Sigma\text{-}\Met$ which can be presented by a set of $\lambda$-basic equations.
\end{definition}

\begin{example}
Generalized varieties are a proper super-class of $\lambda$-basic varieties.
For example, the generalized variety $\Sigma\text{-}\Met$ of Example~\ref{E:bin} (with no equations) is not a $\lambda$-basic variety.
To see this, let $\Tmon$ be the monad of free $\Sigma$-algebras.
We have seen in Proposition~\ref{P:sur} that for every $\lambda$-basic variety the corresponding monad preserves surjective morphisms.
This is not true for generalized varieties.
In the above example let $A$ be the discrete space on $\{ x, y \}$ and $B$ be the two-element space with distance $\frac{1}{2}$ (see Example~\ref{E:bin2}).
For the morphism $\id: A \to B$ the map $T \id$ is not surjective.
\end{example}

\begin{remark}
Analogously to Corollary~\ref{C:term} the free-algebra monad $\Tmon_\Sigma$ of a generalized signature is enriched and $\lambda$-accessible.
Indeed, $T_\Sigma$ is a coproduct of hom-functors of spaces of cardinality less than $\lambda$: one summand for every similarity class of terms of arity $M$.
By Proposition~\ref{P:presen} each $[M,\blank]$ is (enriched and) $\lambda$-accesible.
\end{remark}

\begin{proposition}
\label{P:gen1}
Every generalized variety $\Vvar$ has free algebras.
The corresponding monad $\Tmon_\Vvar$ on $\Met$ is enriched and $\lambda$-accessible.
\end{proposition}

\begin{proof}
We prove below that $\Vvar$ is closed under products, subalgebras, and $\lambda$-directed colimits in $\Sigma\text{-}\Met$.
Then $\Vvar$ has free algebras; this is analogous to Corollary~\ref{C:left}.
Moreover, it follows that the canonical monad morphism $\Tmon_\Sigma \to \Tmon_\Vvar$ has surjective components, thus $\Tmon_\Vvar$ is enriched (since $\Tmon_\Sigma$ is).
Since $\Tmon_\Sigma$ is $\lambda$-accessible, closure under $\lambda$-directed colimits implies that $\Tmon_\Vvar$ is also $\lambda$-accessible.
\begin{enumerate}
	\item Closure under products: If algebras $A_i$ ($i \in I$) satisfy $M \vdash l =_\eps r$, then so does $A = \prod_{i \in I} A_i$.
	Indeed, for each interpretation $f = \langle f_i \rangle_{i \in I}: M \to A$ we know that $d(f_i^\sharp(l),f_i^\sharp(r)) \leq \eps$.
	Moreover $f^\sharp = \langle f_i^\sharp \rangle$ since every projection merges both sides.
	This proves
	\[
	d(f^\sharp(l),f^\sharp(r)) = \sup_{i \in I} d(f_i^\sharp(l),f_i^\sharp(r)) \leq \eps.
	\]
	
	\item Closure under subalgebras: Let $B$ be a subalgebra of $A$ via an isometric embedding $m : B \to A$.
	If $A$ satisfies $M \vdash l =_\eps r$, so does $B$.
	Indeed, for every interpretation $f : M \to B$ we have $m \comp f^\sharp = (m \comp f)^\sharp$ since $m$ is a homomorphism.
	We know that $d((m \comp f)^\sharp(l),(m \comp f)^\sharp(r)) \leq \eps$.
	Since $m$ preserves distances, this implies $d(f^\sharp(l),f^\sharp(r)) \leq \eps$.
	
	\item Closure under $\lambda$-directed colimits: Let $A_i$ ($i \in I$) be algebras of a directed diagram over $(I,\leq)$ with colimit $A$, with each of $A_i$ satisfying $M \vdash l =_\eps r$.
	For every interpretation $f: M \to A$, since the metric space $M$ is $\lambda$-presentable (Proposition~\ref{P:presen}), there exists $i \in I$ and a nonexpanding map $\ol{f} : M \to A_i$ with $f = c_i \comp \ol{f}$.
	We have $f^\sharp = c_i \comp \ol{f}^\sharp$ because $c_i$ is a homomorphism.
	And we know that $d(\ol{f}^\sharp(l),\ol{f}^\sharp(r)) \leq \eps$.
	As $c_i$ is nonexpanding, this proves $d(f^\sharp(l),f^\sharp(r)) \leq \eps$.
\end{enumerate}
\end{proof}

\begin{proposition}
\label{P:gen}
Every generalized variery $\Vvar$ is concretely isomorphic to $\Met^{\Tmon_\Vvar}$, the Eilenberg-Moore category of its free-algebra monad.
\end{proposition}

The proof is analogous to that of Proposition~\ref{P:acc}.

We now consider the converse direction: from accessible monads to generalized varieties.
For this we use the work of Kelly and Power that we first shortly recall.

\begin{remark}
In Section~5 of~\cite{kelly+power:adjunctions} a presentation of enriched finitary monads on $\A$ by equations is exhibited.
Here $\A$ is a locally finitely presentable category (in the enriched sense).
All of that section immediately generalizes to enriched $\lambda$-accesible monads on a locally $\lambda$-presentable category.
For our given (uncountable) cardinal $\lambda$ this specializes to $\Met$ (see Corollary~\ref{C:present}) as follows:
\begin{enumerate}
	\item Recall $\Met_\lambda$ from Notation~\ref{N:lambda} and denote by
	\[
	\ol{K} : \ol{\Met_\lambda} \hookrightarrow \Met_\lambda
	\]
	the full embedding of the discrete category on the same objects ($\N$ is used in place of $\ol{\Met_\lambda}$ in~\cite{kelly+power:adjunctions}).
	We work with the functor category $[\ol{\Met_\lambda},\Met]$.
	An object $B : \ol{\Met_\lambda} \to \Met$ can be considered as a collection of metric spaces $B(M)$ indexed by $M \in \ol{\Met_\lambda}$.
	In this sense every generalized signature is simply an object $B$ of $[\ol{\Met_\lambda},\Met]$ with all spaces $B(M)$ discrete.

	\item The functor category $[\Met_\lambda,\Met]$ is equivalent to the category of all $\lambda$-accessible endofunctors on $\Met$.
	We thus have an obvious forgetful functor
	\[
	W : \Mon_\lambda(\Met) \to [\Met_\lambda,\Met],
	\]
	assigning to a $\lambda$-accessible monad its underlying endofunctor.
	We also have the forgetful functor
	\[
	V: [\Met_\lambda,\Met] \to [\ol{\Met_\lambda},\Met]
	\]
	given by precomposition with $\ol{K}$.
	The composite $V \comp W$ assigns to a monad $\Tmon$ the collection $(TM)_{M \in \Mon_\lambda}$.
	
	\item The functor $V$ has a left adjoint
	\[
	G \adj V
	\]
	given by the left Kan extension along $\ol{K} : \ol{\Met_\lambda} \hookrightarrow \Met_\lambda$:
	\[GB = \Lan{\ol{K}}{B\ol{K}}.
	\]
	It assigns to $B : \ol{\Met_\lambda} \to \Met$ the functor
	\[
	GB = \coprod_{M \in \Met_\lambda} B(M) \tensor [M,\blank].
	\]
\end{enumerate}
\end{remark}

\begin{theorem}[\cite{lack:monadicity}]
The forgetful functor
\[
V \comp W : \Mon_\lambda(\Met) \to [\ol{\Met_\lambda},\Met]
\]
is monadic.
\end{theorem}

\begin{notation}
The left adjoint $F \adj V \comp W$ assigns to $B : \ol{\Met_\lambda} \to \Met$ the free monad on the $\lambda$-accessible endofunctor of $\Met$ corresponding to $GB$.
\end{notation}

\begin{corollary}
Every enriched $\lambda$-accessible monad $\Tmon$ is a coequalizer, in $\Mon_\lambda(\Met)$, of a parallel pair
\begin{equation}
\label{KP1} \tag{KP1}
\begin{tikzcd}
FB' \arrow[r, bend left, "\sigma"] \arrow[r, bend right, swap, "\tau"] & FB \arrow[r, "\rho"] & \Tmon
\end{tikzcd}
\end{equation}
of monad morphisms $\sigma$, $\tau$ (for some $B',B: \ol{\Met_\lambda} \to \Met$).
\end{corollary}

\begin{remark}
\label{R:KP}
Under the adjunction $F \adj V\comp W$ the monad morphisms $\sigma$, $\tau$ correspond to $\sigma',\tau': F \to VWFB$.
Kelly and Power deduce that the category $\Met^\Tmon$ can be identified with the full subcategory of $\Met^{FB}$ on algebras $\alpha : (FB)A \to A$ that satisfy the following quantitative equations
\begin{equation}
\label{KP2} \tag{KP2}
M \vdash \sigma_M'(x) = \tau_M'(x) \text{ for all } M \in \Met_\lambda \text{ and } x \in B'(M).
\end{equation}
Satisfaction means that for every morphism $f : M \to A$ we have
\[
f^\sharp(\sigma_M'(x)) = f^\sharp(\tau_M'(x)) \text{ for } f^\sharp = \alpha \comp (FB)f : (FB)M \to A.
\]
We use this to derive the corresponding presentation by generalized signatures and quantitative equations:
\end{remark}

\begin{theorem}
\label{T:gen}
Every $\lambda$-accessible enriched monad on $\Met$ is the free-algebra monad of a generalized variety.
\end{theorem}

\begin{proof}
\phantom{phantom}
\begin{enumerate}
	\item We first associate with every functor $B : \ol{\Met_\lambda} \to \Met$ a generalized signature $\Sigma$ whose $M$-ary symbols are the elements of $BM$:
	\[
	\Sigma_M = |BM| \text{ for } M \in \Met_\lambda.
	\]
	The free $\Sigma$-algebra $T_\Sigma X$ on a space $X$ is, by Remark~\ref{R:Barr}, the colimit $T_\Sigma X = \colim_{i < \lambda} W_i$ of the $\lambda$-chain $W_i$ where
	\[
	W_{i+1} = \coprod_{M \in \Met_\lambda} \Sigma_M \times [M,X].
	\]
	Whereas the free algebra $(GB)X$ is, by the same Remark, the colimit of a $\lambda$-chain $W_i$ ($i < \lambda$) where
	\[
	W_{i+1}' = \coprod_{M \in \Met_\lambda} B(M) \tensor [M,X].
	\]
	We conclude that the underlying sets of the chains $(W_i)_{i < \lambda}$ and $(W_i')_{i < \lambda}$ are the same, thus
	\[
	|T_\Sigma X| = |(GB)X|.
	\]
	
	\item For every functor $B : \ol{\Met_\lambda} \to \Met$ we now show that $\Met^{FB}$ is concretely equivalent to a generalized variety of $\Sigma$-algebras.
	Since $FB$ is the free monad on $H = GB$, we can work with $H$-algebras in place of $\Met^{FB}$ (Remark~\ref{R:Barr}).
	To give a $H$-algebra on a space $A$ means to give a morphism
	\[
	\coprod_{M \in \Met_\lambda} B(M) \tensor [M,A] \to A.
	\]
	Or, equivalently, a collection of morphisms
	\[
	\alpha_M : B(M) \to [[M,A],A]
	\]
	for $M \in \Met_\lambda$.
	Thus $\alpha_M$ assigns to every $\sigma \in \Sigma_M = |BM|$ an operation
	\begin{equation}
	\label{KP3} \tag{KP3}
	\sigma_A = \alpha_M(\sigma) : [M,A] \to A
	\end{equation}
	of arity $M$.
	Moreover, $\alpha_M$ is nonexpanding: for $\sigma,\tau \in \Sigma_M$ we have
	\[
	d(\sigma,\tau) = \eps \text{ in } B(M) \text{ implies } d(\sigma_A,\tau_A) \leq \eps.
	\]
	In other words, the $\Sigma$-algebra $A$ satisfies the $\lambda$-basic equations
	\begin{equation}
	\label{KP4} \tag{KP4}
	M \vdash \sigma =_\eps \tau \text{ for } M \in \Met_\lambda \text{ and } d(\sigma,\tau) = \eps \text{ in } B(M).
	\end{equation}
	Conversely, every $\Sigma$-algebra $A$ satisfying~\eqref{KP3} corresponds to a unique $H$-algebra on the space $A$ with $\alpha_M$ defined by~\eqref{KP3}.
	We conclude that $\Met^{FB}$ is concretely equivalent to the generalized variety presented by~\eqref{KP4}.
		
	\item Let $\Tmon$ be an eriched $\lambda$-accessible monad.
	In Remark~\ref{R:KP} we have seen that $\Met^\Tmon$ is the subcategory of $\Met^{FB}$ presented by the equations~\eqref{KP2}.
	Due to~(1) both $\sigma_M'(x)$ and $\tau_M'(x)$ are elements of $| T_\Sigma M|$, thus~\eqref{KP2} are $\lambda$-basic equations for $\Sigma$ in the sense of Definition~\ref{D:bass}.
	We conclude that $\Tmon$ is the free-algebra monad of the generalized variety presented by~\eqref{KP2} and~\eqref{KP4}.
\end{enumerate}
\end{proof}

\begin{theorem}
\label{T:main3}
For every uncountable regular cardinal $\lambda$ the following categories are dually equivalent:
\begin{enumerate}
	\item generalized varieties of $\lambda$-ary quantitative algebras (and concrete functors), and
	\item $\lambda$-accessible monads on $\Met$ (and monad morphisms).
\end{enumerate}
\end{theorem}

This follows from Theorem~\ref{T:gen} and Propositions~\ref{P:gen} precisely as in the proof of Theorem~\ref{T:main}.
For $\lambda = \aleph_0$ the corresponding result does not hold: see Example~\ref{E:contra}.

\end{document}